\documentclass[11pt, reqno]{amsart}
\usepackage{amssymb}
\usepackage{amsmath}
\usepackage{enumerate}
\usepackage{mathrsfs}
\usepackage{mathtools}
\usepackage{amsfonts}
\usepackage{color}
\usepackage{vmargin}
\usepackage{amsthm}
\usepackage{caption}
\usepackage{graphicx} 
\usepackage{subcaption}
\usepackage{xcolor}
\usepackage{physics}
\usepackage{bbm}
\usepackage{capt-of}
\usepackage[normalem]{ulem}
\usepackage{comment}
\usepackage{isomath}
\usepackage[makeroom]{cancel}
\usepackage[nopar]{lipsum}
\usepackage{esint}

\usepackage{tikz}
\usepackage{pgfplots}
\pgfplotsset{compat=1.13}

\def\XXint#1#2#3{{\setbox0=\hbox{$#1{#2#3}{\int}$ }
		\vcenter{\hbox{$#2#3$ }}\kern-.6\wd0}}

\usepackage[utf8]{inputenc}
\usepackage[T1]{fontenc}

\DeclareMathOperator*{\diam}{diam}

\DeclareMathOperator{\interior}{int}

\newcommand\restr[2]{{% we make the whole thing an ordinary symbol
  \left.\kern-\nulldelimiterspace % automatically resize the bar with \right
  #1 % the function
  \littletaller % pretend it's a little taller at normal size
  \right|_{#2} % this is the delimiter
  }}

\newcommand{\littletaller}{\mathchoice{\vphantom{\big|}}{}{}{}}

\long\def\symbolfootnote[#1]#2{\begingroup%
	\def\thefootnote{\fnsymbol{footnote}}\footnote[#1]{#2}\endgroup}

\setmarginsrb{20mm}{20mm}{20mm}{20mm}{10mm}{10mm}{10mm}{10mm}

\newtheoremstyle{definition}%    % Name
{3ex}%                          % Space above
{3ex}%                          % Space below
{\upshape}%                      % Body font
{}%                              % Indent amount
{\bfseries}%                     % Theorem head font
{.}%                             % Punctuation after theorem head
{.5em}%                            % Space after theorem head, ' ', or \newline
{\thmname{#1}\thmnumber{ #2}\thmnote{ (#3)}}%  % Theorem head spec (can be left empty, meaning `normal')

\newtheoremstyle{remark}
{}{}{}{}{\bfseries}{.}{.5em}{{\thmname{#1 }}{\thmnumber{#2}}{\thmnote{ (#3)}}}
\theoremstyle{remboldstyle}

\RequirePackage{amsthm}
\newcommand{\brac}[1]{\left({#1}\right)}
\newcommand{\bracl}[1]{\left[#1\right]}

\theoremstyle{definition}
\newtheorem{defi}{Definition}[section]

\newtheorem{tw}[defi]{Theorem}
\newtheorem{lem}[defi]{Lemma}
\newtheorem{cor}[defi]{Corollary}
\newtheorem{obs}[defi]{Observation}

\newtheorem{prop}[defi]{Proposition}
\newtheorem{rem}[defi]{Remark}
\newtheorem{ex}[defi]{Example}
\newtheorem*{tw*}{Theorem}

\newtheorem{twmainA}{Theorem}

\newtheorem{twmainB}{Theorem}

\def\={\hspace{-3mm}&=&\hspace{-3mm}}

\begin{document}
	
\date{}

	\title[\bf Chessboard and level sets of continuous functions]{\bf Chessboard and level sets of continuous functions}
	
	\author{Micha{\l} Dybowski and Przemys{\l}aw G\'{o}rka}%\\
	
	\maketitle
		\begin{abstract}
    We provide the following result and its discrete equivalent: \textit{Let $f \colon I^n \to \mathbb{R}^{n-1}$ be a continuous function. Then, there exist a point $p \in \mathbb{R}^{n-1}$ and a compact subset $S \subset f^{-1}\bracl{\left\{p\right\}}$ which connects some opposite faces of the $n$-dimensional unit cube $I^n$}.
    
    We give an example that shows it cannot be generalized to path-connected sets. Additionally, we show that a version of the Steinhaus Chessboard Theorem and the Brouwer Fixed Point Theorem are simple consequences of this result.
	\end{abstract}
	\bigskip
	
	\noindent
	{\bf Keywords}: topological combinatorics, Steinhaus Chessboard Theorem, Poincaré-Miranda Theorem, connected components, clustered chromatic number, Brouwer Fixed Point Theorem, fibers of continuous functions.
	\medskip
	
	\noindent
	\emph{Mathematics Subject Classification (2020):} 54D05, 54C05, 05C15, 51M99.
	
\section{Introduction}
There are many combinatorial and continuous results concerning the problems in which the structure of some objects is examined in relation to the opposite faces of the $n$-dimensional unit cube $I^n$ or antipodal points of the $n$-dimensional unit sphere $S^n$. The well-known continuous examples are the celebrated Poincaré-Miranda Theorem \cite{miranda}, which is equivalent to the Brouwer Fixed Point Theorem \cite{kulpa}, or the Borsuk-Ulam Theorem \cite{matouvsek}. Among the combinatorial examples, it is essential to mention the Steinhaus Chessboard Theorem \cite{kalejdoskop} and its some generalization \cite[Theorem $1$]{chessboard} provided by Tkacz and Turzański, by means of which the Poincaré-Miranda Theorem can be easily proved \cite[Theorem $3$]{chessboard}, or the Tucker's Lemma \cite{freund} which is a combinatorial analog of the Borsuk-Ulam Theorem. In addition, for related problems we recommend referring to, for instance, \cite{combi6}, \cite{combi1}, \cite{combi2}, \cite{combi3}, \cite{combi5}, \cite{combi4}, \cite{kulpa}, \cite{kulpaparametric}, \cite{waksman}, or \cite{kidawa}.  

We present the following formulations of the Brouwer Fixed Point Theorem, Poincaré-Miranda Theorem, its parametric extension, and a version of the Steinhaus Chessboard Theorem that we use in this paper (cf. \cite[Theorems 1 and 3 and 5]{chessboard}, \cite{kulpa}).

\begin{tw}[Brouwer Fixed Point Theorem]
    Let $n \in \mathbb{N}$. Every continuous function $f \colon I^n \to I^n$ has at least one fixed point i.e. there exists $x_0 \in I^n$ such that $f(x_0) = x_0$.
\end{tw}

\begin{tw}[Poincaré-Miranda Theorem]
    Let $n \in \mathbb{N}$ and $f = (f_1, \ldots, f_n) \colon I^n \to \mathbb{R}^n$ be a~continuous function such that for each $i \in [n]$, it follows that\footnote{See Definition \ref{decompositiondefi}.} $\restr{f_i}{I^n_{i, -}} \le 0$ and $\restr{f_i}{I^n_{i, +}} \ge 0$. Then, there exists a point $x_0 \in I^n$ such that $f(x_0) = (0, \ldots, 0)$.
\end{tw}

\begin{tw}[Parametric extension of the Poincaré-Miranda Theorem]\label{parammira}
    Let $n \in \mathbb{N}$ and $f \colon I^n \times I \to \mathbb{R}^n, f = \brac{f_1, \ldots, f_n}$ be a continuous function such that for each $i \in [n]$, it follows that $\restr{f_i}{I^n_{i, -} \times I} \le 0$ and $\restr{f_i}{I^n_{i, +} \times I} \ge 0$. Then, there exists a connected subset $W \subset f^{-1}\bracl{\left\{0\right\}}$ such that $W \cap \brac{I^n \times \left\{0\right\}} \neq \emptyset \neq W \cap \brac{I^n \times \left\{1\right\}}$.
\end{tw}

\begin{tw}[Version of the Steinhaus Chessboard Theorem]\label{chessboardthm}
Let $n, k \in \mathbb{N}$ and\footnote{See Definition \ref{decompositiondefi}.} $F \colon \mathcal{K}_k^n \to [n]$. Then, there exist $p \in [n]$ and $\mathcal{S} \subset F^{-1}\bracl{\left\{p\right\}}$ such that $\bigcup \mathcal{S}$ connects some opposite faces of $I^n$.
\end{tw}

Some of the aforementioned literature includes results concerning problems related to the opposite faces of $I^n$, where a given continuous or discrete continuous function is considered. These results generally pertain to situations in which certain boundary conditions are imposed on the function, with the exception of the result of Waksman and Wasilewsky \cite{waksman}, which addresses the two-dimensional case of Theorem \ref{contthm} presented below. Nevertheless, there appears to be few comparable results in the existing literature that consider arbitrary continuous or discrete continuous functions without imposing any restrictions on their values. The objective of this paper is to establish results that partially address this issue. 

In this context, we provide the following combinatorial result, which constitutes a~relationship with the version of the Steinhaus Chessboard Theorem, and derive its continuous version, referred to as Theorem \ref{contthm}. 

\begin{twmainA}\label{maincombi}
Let $n \in \mathbb{N}, m \in \mathbb{Z}_+$ be such that $0 \le m \le n-1$. Then, there exists a constant $C_{n, m}>0$ such that the following property holds:
\begin{itemize}
    \item[] Let $k \in \mathbb{N}$ and let $F \colon \mathcal{K}_k^n \to \mathbb{Z}^{n-1}$ be a function such that $\norm{F(K_1) - F(K_2)}_\infty \le 1$ if $\dim\brac{K_1 \cap K_2} \ge m$. Then there exist a $1$-connected subset $P \subset \mathbb{Z}^{n-1}$ with $\abs{P} \le C_{n, m}$ and a~subset $\mathcal{S} \subset F^{-1}\bracl{P}$ such that $\bigcup \mathcal{S}$ connects some opposite faces of $I^n$.
\end{itemize}
\end{twmainA}

While Theorem \ref{chessboardthm} involves at most $n$ colors with any arrangement on the chessboard, our result does not impose any restrictions on the number of colors but does on their arrangement on the chessboard. In this context, the condition imposed on the function $F$ is intended to emulate a discrete continuous function. A particularly intriguing aspect related to this result is the matter of the optimal constants $C_{n, m}$, which we denote by $\widehat{C}_{n, m}$. We obtain partial results concerning this issue (see Section \ref{minimalconst}). Namely, we prove that in dimension two the optimal constants satisfy $\widehat{C}_{2, 0} = \widehat{C}_{2, 1} = 1$. We further show that, in contrast, for dimensions $n\ge 3$ the optimal constants are strictly greater than $1$. Finally, we determine the exact value $\widehat{C}_{3, 0} = 2$.

As a corollary derived from the combinatorial result we shall present the following Poincaré-Miranda type theorem of a purely continuous nature.

\begin{twmainB}\label{contthm}
Let $n \in \mathbb{N}$ and $f \colon I^n \to \mathbb{R}^{n-1}$ be a continuous function. Then, there exist a point $p \in \mathbb{R}^{n-1}$ and a compact subset $S \subset f^{-1}\bracl{\left\{p\right\}}$ which connects some opposite faces of $I^n$.
\end{twmainB}

It turns out the result simply implies the Brouwer Fixed Point Theorem (so also the Poincaré-Miranda Theorem) and the version of the Steinhaus Chessboard Theorem, see Section \ref{sectionconseq}. In this manner, we demonstrate the equivalence between this result and the latter one. To the best of our knowledge, Theorem \ref{contthm} does not possess a direct reference in the literature and may, in fact, be unknown. We emphasize, however, that the case $n=2$ is known (see \cite{waksman}), but the proof relies essentially on the topology of $\mathbb{R}^2$. Let us notice that Theorem \ref{contthm} can be read as a substitute for the parametric extension of the Poincaré-Miranda Theorem (Theorem \ref{parammira}) which does not require any information about the values of the considered continuous function.
\medskip

The remainder of the paper is structured as follows. In Section \ref{prelim} we provide some definitions and technical facts that will be useful in next sections. Section \ref{sectionclustered} is an introduction to the notion of $m$-distance clustered chromatic numbers. Section \ref{sectionmaincombi} is devoted to the proof of Theorem \ref{maincombi} i.e. main combinatorial result. It employs Theorem \ref{chessboardthm} and the result regarding $m$-distance clustered chromatic numbers from Section \ref{sectionclustered}. In Section \ref{minimalconst} we examine the minimum constants that concern Theorem \ref{maincombi}. Section \ref{sectionmainconti} is dedicated to the proof of the continuous result, Theorem \ref{contthm}, that uses Theorem \ref{maincombi} to solve the approximate problem, and then the machinery of Hausdorff distance is used to complete the proof. We also include Example \ref{counterexamplesinus} which shows Theorem \ref{contthm} cannot be generalized to path-connected sets. Finally, in Section \ref{sectionconseq} we provide the proofs of the Brouwer Fixed Point Theorem and the version of the Steinhaus Chessboard Theorem using Theorem \ref{contthm}.

\section{Preliminaries}\label{prelim}

\textbf{Notations and conventions:} Let $\mathbb{Z}^0 = \mathbb{R}^0 = \left\{0\right\}$, $\mathbb{Z}_+ = \mathbb{N} \cup \left\{0\right\}$, $[N] := \left\{1, \ldots, N\right\}$ for $N > 0$, and $[0] := \emptyset$. We set $\sum_{i=n_1}^{n_2} a_i = 0$ if $n_1 > n_2$. For a function $f \colon X \to Y$ and subsets $X_0 \subset X$ and $Y_0 \subset Y$, symbols $f\bracl{X_0}$ and $f^{-1}\bracl{Y_0}$ mean the image of $X_0$ and the preimage of $Y_0$ under $f$, respectively. For a topological space $X$ and subsets $Y, A \subset X$, symbol $\partial A$ means the boundary of $A$ and symbol $\partial_Y A$ means the boundary of $A$ in the topology of $Y$. Symbols $\interior{A}$ and $\overline{A}$ mean the interior of $A$ and the closure of $A$, respectively. By $\dim X$ we denote the topological dimension of space $X$. We use that notion to simplify some statements. For $n \in \mathbb{N}$, $x \in \mathbb{R}^n$ and $r>0$, by $B(x, r)$ we denote the open ball in $\mathbb{R}^n$ with the center at the point $x$ and radius $r>0$ considered with the standard norm. The closed ball is denoted by $\overline{B}(x, r)$. By $B_{l^\infty}(x, r)$ we denote the open ball considered with the $l^\infty$ norm so $B_{l^\infty}(x, r)$ is an open cube. If $A \subset \mathbb{R}^n$ for some $n \in \mathbb{N}$, then by $\diam A$ we denote the diameter of the set $A$ in the standard norm. For $n \in \mathbb{N}$, whenever we consider $\mathbb{Z}^n$ as the space, we are dealing with the $l^\infty$ norm. If $a, b \in \mathbb{Z}$ and $b \neq 0$, then $a \pmod{b}$ means the remainder of dividing $a$ by $b$.

\begin{defi}\label{decompositiondefi}
    Let $n, k\in \mathbb{N}$ and $I^n = [0, 1]^n$ be the $n$-dimensional unit cube. We define a family of cubes
\begin{align*}
\mathcal{K}_k^n := \left\{\prod_{s=1}^n \left[\frac{i_s-1}{k}, \frac{i_s}{k}\right] \colon i_s \in [k]\right\}
\end{align*}
that is the division of $I^n$ into $k^n$ cubes, and a bijection $j_{n, k} \colon [k]^{n} \to \mathcal{K}^n_{k}$ such that 
\begin{align*}
j_{n, k}(i_1, \ldots, i_n) = \prod_{s=1}^n \bracl{\frac{i_s-1}{k}, \frac{i_s}{k}}.
\end{align*}
For $i \in [n]$ we denote $i$th opposite faces of the $n$-dimensional unit cube as 
\begin{align*}
    I_{i, -}^n = \left\{z \in I^n \colon z_i = 0\right\}, \, I_{i, +}^n = \left\{z \in I^n \colon z_i = 1\right\}.
\end{align*}

Moreover we say that a subset $S \subset I^n$ \texttt{connects $i$th opposite faces of $I^n$} if $S$ is connected and $S \cap I_{i, -}^n  \neq \emptyset \neq S \cap I_{i, +}^n$. Since the specific $i$ is often irrelevant in our considerations, we say that the subset $S \subset I^n$ \texttt{connects some opposite faces of $I^n$} if $S$ connects $i$th opposite faces of $I^n$ for some $i \in [n]$.
\end{defi}

\begin{defi}

Let $(X, d)$ be a metric space, $N>0$ and $r > 0$. We say a sequence $p \colon [N] \to X$ is an~\texttt{$r$-path} if $d(p_j, p_{j+1}) \le r$ for all indices $j \in [N-1]$.

We say that subset $S \subset X$ is an \texttt{$r$-connected} set if for every $x, y \in S$ there exist $N > 0$ and an~$r$-path $p \colon [N]\to S$ such that $p_1 = x$ and $p_N = y$. The set $S \subset X$ is an \texttt{$r$-connected component of $X$} if it is a maximal $r$-connected set i.e. there is no $r$-connected set $S' \subset X$ such that $S \subsetneq S'$.  
\end{defi}

Let us notice that every connected set is $r$-connected for all $r >0$ but the converse implication does not hold in general i.e. there exist sets which are $r$-connected for all $r>0$ but not connected. As an~example, we can take the set of rational numbers $\mathbb{Q}$ endowed with the standard topology. 

\begin{obs}\label{obsforj}
Let $n, k \in \mathbb{N}$, $F \colon \mathcal{K}_k^n \to \mathbb{Z}^{n-1}$ and $m \in \mathbb{Z}_+$ be such that $0 \le m \le n-1$. Let $\mathcal{S} \subset \mathcal{K}_k^n$ and $i, i' \in [k]^n$. Then,\footnote{These facts should be visually clear but for the sake of completeness we provide a short argument.}
\begin{enumerate}
    \item the set $\bigcup \mathcal{S}$ is connected if and only if for every $K, K' \in \mathcal{S}$ there exists a sequence $p \colon \bracl{N} \to \mathcal{S}$ for some $N>0$ such that $p_1 = K$, $p_N = K'$ and $p_s \cap p_{s+1} \neq \emptyset$ for all $s \in \bracl{N-1}$;
    \item $j_{n, k}(i) \cap j_{n, k}(i') \neq \emptyset$ if and only if $\norm{i - i'}_\infty \le 1$;
    \item the set $\bigcup \mathcal{S}$ is connected if and only if the set $j^{-1}_{n, k}\bracl{\mathcal{S}}$ is $1$-connected;
    \item if $\bigcup \mathcal{S}$ connects some opposite faces of $I^n$, then $\abs{\mathcal{S}} \ge k$;
    \item if $\dim\brac{K \cap K'} \ge r$ for some $K, K' \in \mathcal{K}_k^n$ and integer $0 \le r \le n-1$, then there exists $K_0 \in \mathcal{K}_k^n$ such that $\dim\brac{K \cap K_0} \ge r+1$ and $\dim\brac{K' \cap K_0} \ge n-1$;
    \item if for all $K, K' \in \mathcal{K}_k^n$ such that $\dim\brac{K \cap K'} \ge m$ we have $\norm{F\brac{K'} - F\brac{K}}_\infty \le 1$, then for all $K, K' \in \mathcal{K}_k^n$ such that $\dim\brac{K \cap K'} \ge 0$ it follows that $\norm{F\brac{K'} - F\brac{K}}_\infty \le m+1$.
\end{enumerate}
\end{obs}

\begin{proof}
$(1)$: Let us assume that $\bigcup \mathcal{S}$ is connected and nonempty and fix $K_0 \in \mathcal{S}$. Let $\mathcal{A}$ be the family of all cubes $K \in \mathcal{S}$ such that there exists a sequence $p \colon [N] \to \mathcal{S}$ for some $N>0$ with $p_1 = K_0$, $p_N = K$ and $p_s \cap p_{s+1} \neq \emptyset$ for all $s \in \bracl{N-1}$. It is enough to prove that $\mathcal{A} = \mathcal{S}$. Both sets $\bigcup \mathcal{A}$ and $\bigcup \brac{\mathcal{S} \setminus \mathcal{A}}$ are closed and mutually disjoint by the definition of $\mathcal{A}$. Moreover, since they sum up to the connected set $\bigcup \mathcal{S}$ and $\bigcup \mathcal{A}$ is nonempty, then $\bigcup \brac{\mathcal{S} \setminus \mathcal{A}}$ has to be empty so $\mathcal{A} = \mathcal{S}$. The converse implication is obvious since the assumption entails the set $\bigcup \mathcal{S}$ is path-connected so also connected.

$(2)$: The intersection $j_{n, k}(i) \cap j_{n, k}(i')$ is nonempty if and only if the intersection $\bracl{(i_s-1)/k, i_s/k} \cap \bracl{(i'_s-1)/k, i'_s/k}$ is nonempty for all $s \in [n]$ but it is equivalent to the fact that $\abs{i_s - i'_s} \le 1$ for all $s \in [n]$.

$(3)$: Let us assume that $\bigcup \mathcal{S}$ is connected and nonempty and take $i, i' \in j_{n, k}^{-1}\bracl{\mathcal{S}}$. Then, by $(1)$ there exists a sequence $p \colon \bracl{N} \to \mathcal{S}$ for some $N>0$ such that $p_1 = j_{n, k}\brac{i}$, $p_N = j_{n, k}\brac{i'}$ and $p_s \cap p_{s+1} \neq \emptyset$ for all $s \in \bracl{N-1}$. Let $\widehat{p} \colon \bracl{N} \to j_{n, k}^{-1}\bracl{\mathcal{S}}$ be defined as $\widehat{p}_s = j_{n, k}^{-1}\brac{p_s}$ for $s \in \bracl{N}$. Then $j_{n, k}\brac{\widehat{p}_{s+1}} \cap j_{n, k}\brac{\widehat{p}_s} \neq \emptyset$ for $s \in \bracl{N-1}$ so $\norm{\widehat{p}_{s+1} - \widehat{p}_s}_\infty \le 1$ from $(2)$. Therefore, $\widehat{p}$ is a $1$-path and $\widehat{p}_1 = i$, $\widehat{p}_N = i'$ so it proves the set $j_{n, k}^{-1}\bracl{\mathcal{S}}$ is $1$-connected. A proof of the converse implication proceeds similarly using $(1)$.

$(4)$: It is obvious since each cube from $\mathcal{K}_k^n$ has the side length of $1/k$ and $I^n$ has the side length of $1$.

$(5)$: Let $K = \prod_{i=1}^n I_i$ and $K' = \prod_{i=1}^n I_i'$ for some intervals $I_i, I_i' \subset \mathbb{R}$ of length $1/k$ where $i \in [n]$. Since $\dim\brac{K \cap K'} \ge r$, then $I_{i_s} = I_{i_s}'$ for $r$ distinct indices $i_s \in [n]$ where $s \in [r]$. Let $i_0 \in [n]$ be such that $i_0 \neq i_s$ for any $s \in [r]$. We define $\widehat{I}_{i_0} = I_{i_0}$ and $\widehat{I}_i = I_i'$ for $i \in [n] \setminus \left\{i_0\right\}$. Then, the cube $K_0 = \prod_{i=1}^n \widehat{I}_i$ meets the required conditions. 

$(6)$: The case $m=0$ is trivial so we can assume that $m>0$. Let $K, K' \in \mathcal{K}_k^n$ be such that $\dim\brac{K \cap K'} \ge 0$. Using observation $(5)$ repeatedly starting from $r=0$ and going through all $r$ to $r = m-1$, we conclude there exists a sequence $p \colon [m+2] \to \mathcal{K}_k^n$ such that $p_1 = K'$, $p_{m+2} = K$, $\dim\brac{p_s \cap p_{s+1}} \ge n-1$ for $s \in [m]$ and $\dim\brac{p_{m+1} \cap p_{m+2}} \ge m$. Thus, $\norm{F\brac{p_s} - F\brac{p_{s+1}}}_\infty \le 1$ for all $s \in [m+1]$ so $\norm{F\brac{K'} - F\brac{K}}_\infty = \norm{F\brac{p_1} - F\brac{p_{m+2}}}_\infty \le m+1$ from the triangle inequality. 
\end{proof}

The following Proposition \ref{toconstC21} plays a key role in Section \ref{minimalconst}.

\begin{prop}\label{toconstC21}
For $n \in \mathbb{N}$ let $S_0 \subset \mathbb{R}^n$ be a closed and connected set and $S \subset \mathbb{R}^n$ be a connected component of $\mathbb{R}^n \setminus S_0$. Then,
\begin{enumerate}
\item[(i)] the set $\partial S$ is connected;
\item[(ii)] $\partial S \subset S_0$;
\item[(iii)] for every family $\mathcal{U}$ of open and connected sets such that $\mathbb{R}^n \setminus S \subset \bigcup \mathcal{U}$ and $U \cap \brac{\mathbb{R}^n \setminus S} \neq \emptyset$ for all $U \in \mathcal{U}$ the set $S \cap \bigcup \mathcal{U}$ is open and connected.
\end{enumerate}
\end{prop}

\begin{proof}
Let us start by providing some useful facts.

\begin{lem}[{\cite[Section III Theorem $5$]{kuratowski}}]\label{Y}
Let $X$ be a connected space. If $S_0$ is a connected set and $S$ is a connected component of $X \setminus S_0$, then $X \setminus S$ is connected.
\end{lem}

\begin{tw}[{\cite{Czarnecki}}]\label{Z}
Let $n \in \mathbb{N}$ and $U \subset \mathbb{R}^n$ be an open and connected set. Then $\mathbb{R}^n \setminus U$ is connected if and only if $\partial U$ is connected.
\end{tw}

The following result is probably well-known but for the sake of completeness we provide a short proof.

\begin{tw}\label{vietoris}
For $n \in \mathbb{N}$ let $U, V \subset \mathbb{R}^n$ be open and connected sets such that $U \cup V = \mathbb{R}^n$. Then $U \cap V$ is connected.
\end{tw}

\begin{proof}
It is an easy consequence of the Mayer-Vietoris Theorem for Reduced Homology \cite[Corollary $6.4$]{rotman}. Indeed, we can assume that $U \cap V \neq \emptyset$ so there is an exact sequence
\begin{align*}
\widetilde{H}_1\brac{\mathbb{R}^n} \longrightarrow \widetilde{H}_0\brac{U \cap V} \longrightarrow \widetilde{H}_0\brac{U} \oplus \widetilde{H}_0\brac{V}.
\end{align*}
Since $\mathbb{R}^n$ is simply connected and sets $U$ and $V$ are path-connected we obtain $\widetilde{H}_1\brac{\mathbb{R}^n} = \widetilde{H}_0\brac{U} = \widetilde{H}_0\brac{V} = 0$ so $0 \longrightarrow \widetilde{H}_0\brac{U \cap V} \longrightarrow 0$. Thus $\widetilde{H}_0\brac{U \cap V} = 0$ so $U \cap V$ is path-connected.
\end{proof}

\begin{lem}\label{conn}
Let $X$ be a topological space, $S_0 \subset X$ be a connected set and $\mathcal{S}$ be an arbitrary family of connected subsets of $X$ such that $S_0 \subset \bigcup \mathcal{S}$ and $K \cap S_0 \neq \emptyset$ for all $K \in \mathcal{S}$. Then, $\bigcup \mathcal{S}$ is connected.
\end{lem}

\begin{proof}
We can assume that $\mathcal{S} \neq \emptyset$ so $S_0 \neq \emptyset$. Let $\mathcal{S}'$ be the family of all sets $K \cup S_0$ where $K \in \mathcal{S}$. Since $S_0$ is connected and $\mathcal{S}$ consists of connected sets with nonempty intersection with $S_0$, then every set from $\mathcal{S}'$ is connected. Moreover $S_0 \subset \bigcap \mathcal{S}'$ so the set $\bigcap \mathcal{S}'$ is nonempty and thus $\bigcup \mathcal{S}'$ is connected. Since $\bigcup \mathcal{S}' = \bigcup \mathcal{S}$, then $\bigcup \mathcal{S}$ is connected.
\end{proof}

Now we are in position to prove Proposition \ref{toconstC21}. Let us first indicate that since $\mathbb{R}^n \setminus S_0$ is open and $\mathbb{R}^n$ is locally connected, then $S$ is open. Moreover $\mathbb{R}^n \setminus S$ is connected by Lemma \ref{Y}.

    (i): Since $S$ is open and connected with the connected complement, then $\partial S$ is connected by Theorem \ref{Z}.
    
    (ii): Let us suppose that $\partial S \not\subset S_0$. Thus $\partial S \cap \brac{\mathbb{R}^n \setminus S_0} \neq \emptyset$. Let us take $x_0 \in \partial S \cap \brac{\mathbb{R}^n \setminus S_0}$. Then $S \cup \left\{x_0\right\} \subset \mathbb{R}^n \setminus S_0$ is connected. Moreover $S$ is open so $S \cap \partial S = \emptyset$ and then $S \subsetneq S \cup \left\{x_0\right\}$. It contradicts the fact that $S$ is a connected component of $\mathbb{R}^n \setminus S_0$.
    
    (iii): Since $\mathbb{R}^n \setminus S$ is connected, then $\bigcup \mathcal{U}$ is connected by virtue of Lemma \ref{conn}. Therefore, the sets $S$ and $\bigcup \mathcal{U}$ are open and connected and $S \cup \bigcup \mathcal{U} = \mathbb{R}^n$ so $S \cap \bigcup \mathcal{U}$ is connected by Theorem \ref{vietoris}.
\end{proof}

For the convenience of the reader we recall the notion of the Hausdorff distance that will be useful in the proof of Theorem \ref{contthm}. Let $(X, d)$ be a metric space. By $\mathfrak{C}(X)$ we denote the family of all compact and nonempty subsets of $X$. For $A, B \in \mathfrak{C}(X)$ we define the Hausdorff distance between $A$ and $B$ as follows:
\begin{align*}
d_H(A,B)= \inf \left\{\varepsilon>0 \colon\, A \subset \brac{B}_{\varepsilon}, \, B \subset \brac{A}_{\varepsilon} \right\},
\end{align*}
where 
\begin{align*}
\brac{A}_{\varepsilon} = \bigcup_{a \in A}B_d\brac{a, \varepsilon}
\end{align*}
and $B_d(a, \varepsilon)$ is the open ball in $(X, d)$ with center at the point $a$ and radius $\varepsilon$. It is well-known that $(\mathfrak{C}(X), d_H)$ is a metric space. Moreover, $(\mathfrak{C}(X), d_H)$ inherits some properties from $(X, d)$. For instance, we have the following fact.

\begin{tw}[{\cite[Proposition $7.3.7$, Theorem $7.3.8$]{burago}}]\label{blaschke}
    Let $X$ be a metric space.
    \begin{enumerate}
        \item If $X$ is complete, then $\mathfrak{C}(X)$ is complete.
        \item If $X$ is compact, then $\mathfrak{C}(X)$ is compact.
    \end{enumerate}
\end{tw}

The following observation is an easy and well-known consequence of the definition of the Hausdorff distance.

\begin{obs}\label{hausdorff}
Let $(X, d)$ be a metric space and let there be given a sequence of sets $A_n \in \mathfrak{C}(X)$ convergent to some set $A \in \mathfrak{C}(X)$ in Hausdorff sense. Then
\begin{enumerate}
\item for every $a \in A$ there exists a sequence of points $a_n \in A_n$ convergent to $a$;
\item every sequence of points $a_n \in A_n$ has a subsequence $a_{n_m}$ convergent to some point $a \in A$.
\end{enumerate}
\end{obs}

\begin{tw}[{\cite[Section XI, Exercise $1$]{whyburn}}]\label{golab}
    Let $(X, d)$ be a metric space and $A \in \mathfrak{C}(X)$. If there exists a sequence of connected sets $A_n \in \mathfrak{C}(X)$ convergent to $A$ in Hausdorff sense, then $A$ is connected.
\end{tw}

\begin{tw}[{\cite[Theorem $2.29$]{lee}}]\label{extending}
    Let $n \in \mathbb{N}$ and $K \subset \mathbb{R}^n$ be a closed set. Then, there exists a~smooth function $f \colon \mathbb{R}^n \to \mathbb{R}$ such that $f^{-1}\bracl{\left\{0\right\}} = K$. 
\end{tw}
%\smallskip

\section{$m$-distance clustered chromatic number}\label{sectionclustered}
The following definition is a generalization based on the concept of \texttt{clustered chromatic number} (see \cite{clusterednumber} for more information) which comes from the graph theory.
\begin{defi}
    Let $G$ be a graph, $V$ be the set of its vertices and $n, m \in \mathbb{N}$. We say that a function $f \colon V \to [n]$ is an \texttt{$m$-distance clustered $n$-coloring of $G$} if there exists a constant $K_{\star} > 0$ such that for every $i \in [n]$ all connected components of the subgraph induced by the set $f^{-1}\bracl{\left\{i\right\}}$ have at most $K_{\star}$ vertices and the distance between any two of these connected components is bigger than $m$.
    
    We define an \texttt{$m$-distance clustered chromatic number} of $G$ as the minimum natural number $n$ for which there exists an $m$-distance clustered $n$-coloring of $G$. If there is no such number $n$, we set the $m$-distance clustered chromatic number of $G$ to infinity. We denote the $m$-distance clustered chromatic number of $G$ by $\chi_{\star, m}(G)$. If $\chi_{\star, m}(G)$ is finite, then by $\widehat{K}_{\star, m}(G)$ we denote the minimum constant $K_{\star}$ from the definition of the $m$-distance clustered $\chi_{\star, m}(G)$-coloring of~$G$.
\end{defi}

Obviously, the definition makes the most sense if $G$ is a graph with an infinite number of vertices since for graphs $G$ with a finite (but nonzero) number of vertices we have $\chi_{\star, m}\brac{G} = 1$ for every $m \in \mathbb{N}$. Let us stress that this notion does not require from a graph to be connected. Nevertheless, we shall apply it only to the connected graphs. The definition of clustered chromatic number appearing in the literature (\cite{clusterednumber}) coincides with our definition for $m=1$.
    
    \begin{obs}\label{clustobs}
        $\chi_{\star, m_1}(G) \le \chi_{\star, m_2}(G)$ for $m_1 \le m_2$.
    \end{obs}

\begin{ex}
 ($1$): Let $V$ be an arbitrary infinite set and we fix $v_0 \in V$. We define a set $E = \left\{\left\{v, v_0\right\}\colon\; v \in V \setminus \left\{v_0\right\}\right\}$ and $G$ be a graph with the set $V$ of vertices and the set $E$ of edges. We claim that $\chi_{\star, 1}\brac{G} = 2$ and $\chi_{\star, m}\brac{G} = \infty$ for $m > 1$.
 
 Obviously, $\chi_{\star, 1}\brac{G} > 1$ since $V$ is infinite and $G$ is connected. To show that $\chi_{\star, 1}\brac{G} \le 2$ it is enough to define a $1$-distance clustered $2$-coloring $f \colon V \to [2]$ of $G$ as follows: $f^{-1}\bracl{\left\{1\right\}} = V \setminus \left\{v_0\right\}$ and $f^{-1}\bracl{\left\{2\right\}} = \left\{v_0\right\}$. Let us fix $m>1$. Since $\chi_{\star, m}\brac{G} > 1$ and the distance between any two vertices is less than or equal to $2$, then $\chi_{\star, m}\brac{G} = \infty$.
 \vspace{10pt}
 
 ($2$): The following example shows that the difference between $\chi_{\star, 1}(G)$ and $\chi_{\star, 2}(G)$ can be arbitrarily big even if $\chi_{\star, 2}(G)$ is finite i.e. for every $k\ge 2$ we can construct a graph $G_k$ such that $\chi_{\star, 1}(G_k) = 2$ and $\chi_{\star, 2}(G_k) = k+1$.
 
 Let $k \ge 2$, $V_0 = \left\{0\right\}$ and $V_n = [k]^n$ for all $n \in \mathbb{N}$, and $V = \bigcup_{n=0}^\infty V_n$. Let $\pi\colon V \to V \cup \left\{-1\right\}$ be defined as follows: $\pi(0) = -1$, $\pi(v) = 0$ if $v \in V_1$ and $\pi(v) = (v_1, \ldots, v_{n-1})$ if $v = (v_1, \ldots, v_n)$ for some $n \ge 2$. Moreover for $v \in V$ we put $N_{\star}(v) := \pi^{-1}\bracl{\left\{v\right\}}$. Let $E$ be the set of all subsets $\left\{v, v'\right\} \subset V$ such that $\pi(v) = v'$ or $\pi(v') = v$. Finally, let $G_k$ be the graph with the set $V$ of vertices and the set $E$ of edges (see Figure \ref{toclusteredexample1} for the case $k=2$).

   \begin{figure}[h!]
\centering
    \includegraphics[width=0.75\linewidth]{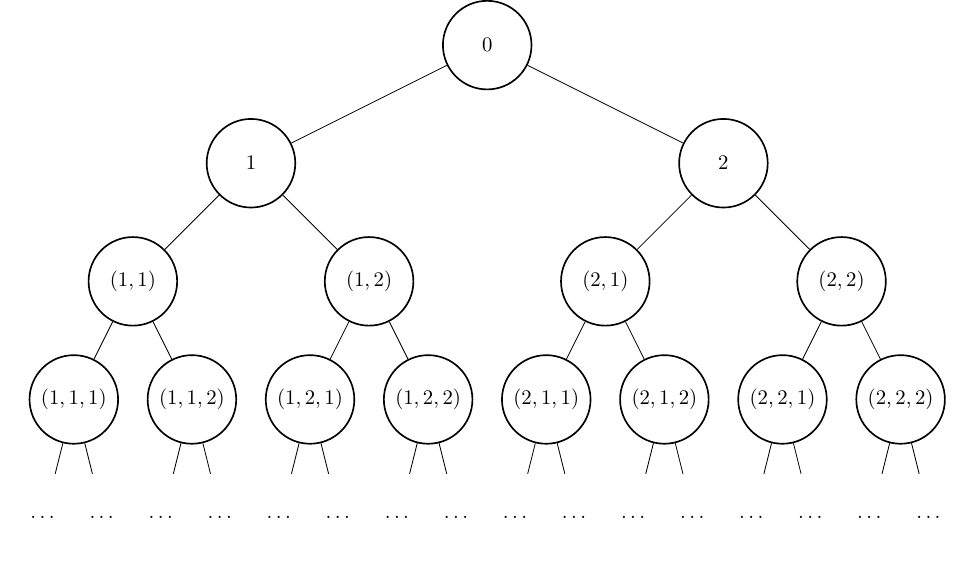}
     \caption{Graph $G_2$.}
     \label{toclusteredexample1}
\end{figure}
 
 Obviously, $\chi_{\star, 1}\brac{G_k} > 1$ since $V$ is infinite and $G_k$ is connected. To show that $\chi_{\star, 1}\brac{G_k} \le 2$ we define a $1$-distance clustered $2$-coloring $f \colon V \to [2]$ of $G_k$ as follows: $f^{-1}\bracl{\left\{1\right\}} = \bigcup_{k=0}^\infty V_{2k}$ and $f^{-1}\bracl{\left\{2\right\}} = \bigcup_{k=0}^\infty V_{2k + 1}$. Then, for both $i \in [2]$ all connected components of the subgraph induced by the set $f^{-1}\bracl{\left\{i\right\}}$ are singletons. 
 
  Let us suppose that $\chi_{\star, 2}(G_k) \le k$ and let $f \colon V \to [k]$ be a $2$-distance clustered $k$-coloring of $G_k$. Let $i \in [k]$ be such that $f^{-1}\bracl{\left\{i\right\}} \neq \emptyset$ and let $S$ be any connected component of the subgraph induced by the~set $f^{-1}\bracl{\left\{i\right\}}$. Since $S$ is finite, there is maximum $n \ge 0$ such that $S \cap V_n \neq \emptyset$. Let $v \in S \cap V_n$. Obviously $\abs{N_{\star}(v)} = k$ and from the maximality of $n$ we have $f^{-1}\bracl{\left\{i\right\}} \cap N_{\star}(v) = \emptyset$. Hence $f\bracl{N_{\star}(v)} \subset [k] \setminus \left\{i\right\}$, but $\abs{N_{\star}(v)} = k$ so there exist two vertices $v_1, v_2 \in N_{\star}(v)$ and $i_0\in [k] \setminus \left\{i\right\}$ such that $f(v_1) = f(v_2) = i_0$. These vertices belong to different connected components of the subgraph induced by $f^{-1}\bracl{\left\{i_0\right\}}$ and the distance between them equals $2$ so it contradicts the definition of $f$.

  To show that $\chi_{\star, 2}\brac{G_k} \le k+1$, we define $f \colon V \to [k+1]$ inductively. Clearly, the graph $G_k$ can be viewed as a tree. Our goal is to assign to the children vertices of the root vertex the same color as that root vertex, and then assign different colors to their children vertices. We then continue this process by treating the grandchildren vertices of the root vertex as new root vertices. It is helpful to refer to Figure~\ref{toclusteredexample2} for a clearer understanding of the construction. We put $f(v) := 1$ for all $v \in V_0 \cup V_1$. For $v \in V \setminus V_0$ such that $f(v)$ is already defined we put $f(v') := f(v)$ for all $v' \in N_{\star}(v)$ if $f(v) \neq f \circ \pi(v)$, and we assign each color except color $f(v)$ to exactly one $v' \in N_{\star}(v)$ in any order if $f(v) = f \circ \pi(v)$. One can easily convince oneself that $f$ is a $2$-distance clustered $(k+1)$-coloring of $G_k$.
  \begin{figure}[h!]
\centering
    \includegraphics[width=0.8\linewidth]{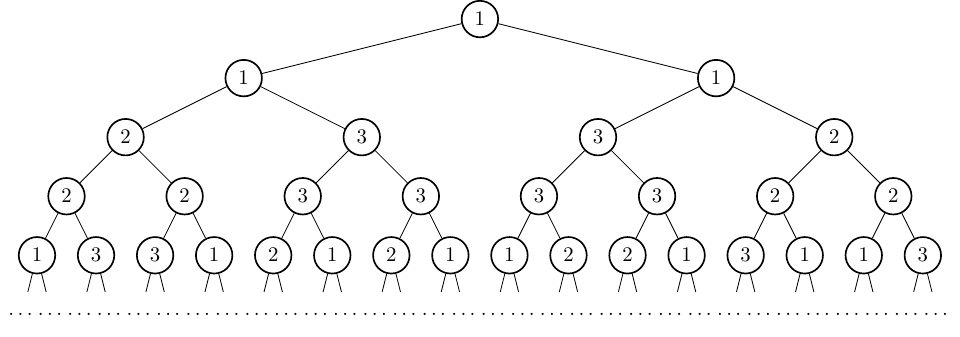}
     \caption{A $2$-distance clustered $3$-coloring of $G_2$.}
     \label{toclusteredexample2}
\end{figure}
 
 ($3$): Below, in Proposition \ref{graph} in particular we show that $\chi_{\star, m}(G)$ can reach the same fixed and finite value for all $m \in \mathbb{N}$.
\end{ex}

Let $n$ be a natural number and $G_n^\infty$ be the graph with the set of vertices $V = \mathbb{Z}^n$ and the set of edges $E_\infty = \left\{\left\{x, y\right\} \subset V \colon\; \norm{x - y}_\infty = 1\right\}$. The following Proposition \ref{graph} plays a key role in Section \ref{sectionmaincombi}.

\begin{prop}\label{graph}
Let $n, m \in \mathbb{N}$. Then $\chi_{\star, m}(G_n^\infty) = n+1$ and $\widehat{K}_{\star, m}\brac{G_n^\infty} \le n! \,m^n$.
\end{prop}

\begin{proof}
Firstly, we shall prove that $\chi_{\star, m}(G_n^\infty) > n$. Conversely, let us suppose that $\chi_{\star, m}(G_n^\infty) \le n$. Then $\chi_{\star, 1}(G_n^\infty) \le n$ by Observation \ref{clustobs} so there exists a constant $K_{\star} > 0$ and a function $f \colon \mathbb{Z}^n \to [n]$ such that for every $i \in [n]$ all $1$-connected components of $f^{-1}\bracl{\left\{i\right\}}$ have cardinalities at most $K_{\star}$. Let $M = K_{\star} + 1$ and $F = f \circ j_{n, M}^{-1} \colon \mathcal{K}_M^n \to [n]$. Then, by virtue of Theorem \ref{chessboardthm} there exist $p \in [n]$ and $\mathcal{S} \subset F^{-1}\bracl{\left\{p\right\}}$ such that $\bigcup \mathcal{S}$ connects some opposite faces of $I^n$. From Observation \ref{obsforj} $(4)$ we have $\abs{\mathcal{S}} \ge M$ so $\abs{j_{n, M}^{-1}\bracl{\mathcal{S}}} = \abs{\mathcal{S}} \ge M > K_{\star}$. Moreover $j_{n, M}^{-1}\bracl{\mathcal{S}}$ is $1$-connected from Observation \ref{obsforj} $(3)$. Finally, we have $f \circ j_{n, M}^{-1}\bracl{\mathcal{S}} = F\bracl{\mathcal{S}} = \left\{p\right\}$ so $j_{n, M}^{-1}\bracl{\mathcal{S}} \subset f^{-1}\bracl{\left\{p\right\}}$ and it contradicts the fact that all $1$-connected components of $f^{-1}\bracl{\left\{p\right\}}$ have cardinalities at most $K_{\star}$.

    Now, we will show that $\chi_{\star, m}(G_n^\infty) \le n+1$. Let $\mathcal{V} = \prod_{i=1}^n [im]$. Then $\abs{\mathcal{V}} = n! \,m^n$. Let us define a~square matrix $A = \brac{a_{ij}}_{i, j = 1}^{n}$ as follows:
    \begin{align*}
a_{ij} = \begin{cases}
0 & \text{if } i>j; \\
(i+1)m & \text{if } i=j; \\
m & \text{if } i<j.
\end{cases}
\end{align*}
Let $V^u_k = \mathcal{V} + Au + k\mathbbm{1}_m$ for $u \in \mathbb{Z}^n, k = 0, \ldots, n$ where $\mathbbm{1}_m = \brac{m, m, \ldots, m}^T \in \mathbb{Z}^n$ and let $V_k = \bigcup_{u \in \mathbb{Z}^n} V_k^u$ for $0 \le k \le n$. We have 
\begin{align*}
(Au)_i = \sum_{j=1}^n a_{ij}u_j = (i+1)mu_i + m\sum_{j=i+1}^n u_j    
\end{align*}
for every $i \in [n]$. We will show that for every $t \in \mathbb{Z}^n$ there exists exactly one $0 \le k \le n$ such that $t \in V_k$.

Firstly, we shall show that the family $\left\{V_k\right\}_{k=0}^n$ consists of pairwise disjoint sets. For this purpose, let us suppose that $V_{k_1} \cap V_{k_2} \neq \emptyset$ for some $k_1$ and $k_2$. Then $v^1 + Au^1 + k_1\mathbbm{1}_m = v^2 + Au^2 + k_2\mathbbm{1}_m$ for some $v^1, v^2 \in \mathcal{V}$ and $u^1, u^2 \in \mathbb{Z}^n$. Let us denote $v = v^1 - v^2, u = u^1 - u^2$ and $k = k_1 - k_2$. Then $v + Au + k\mathbbm{1}_m = 0$ so
\begin{align}\label{prop1}
    v_i + (i+1)mu_i + m\sum_{j=i+1}^n u_j + mk = 0
\end{align}
for every $i \in [n]$ and we obtain
\begin{align}\label{prop2}
    v_i - v_{i-1} + im(u_i - u_{i-1}) = 0
\end{align}
for every $2 \le i \le n$.
Let us suppose that $v \neq 0$. Then we can take $i_0 = \min \left\{i \colon v_i \neq 0\right\}$. Since $v^1_1, v^2_1 \in [1, m]$, we have $\abs{v_1} < m$ so $v_1 = 0$ since $m \,|\, v_1$ from $(\ref{prop1})$. Thus $i_0 > 1$ and $v_{i_0} - v_{i_0 - 1} + i_0m(u_{i_0} - u_{i_0 -1}) = 0$ by $(\ref{prop2})$. We have $v_{i_0 - 1} = 0$ and $i_0m \,|\, \brac{v_{i_0} - v_{i_0 - 1}}$ so $i_0m \,|\, v_{i_0}$. On the other hand, $\abs{v_{i_0}} < i_0 m$ so $v_{i_0} = 0$ and it contradicts the definition of $i_0$. Hence $v=0$ and from $(\ref{prop1})$ for $i=n$ we obtain $(n+1)mu_n + mk = 0$ so $(n+1)\,|\,k$. But $\abs{k} \le n$ so $k=0$ and finally $k_1 = k_2$.

We shall show that $\mathbb{Z}^n = \bigcup_{k=0}^n V_k$. Let us fix $t \in \mathbb{Z}^n$. We define 
\begin{align*}
    &v_1 = (t_1 - 1) \pmod{m} + 1, \\&
    v_i = \brac{\brac{t_i - t_{i-1}} \pmod{im} + v_{i-1} -1} \pmod{im} + 1
\end{align*}
 for all $i = 2, \ldots, n$. Then $v_i \in [im]$ for all $i \in [n]$ so $v = (v_1, v_2, \ldots, v_n) \in \mathcal{V}$. Moreover from the definition of $v_i$ we obtain 
 \begin{align}\label{prop3}
 im\,|\, \brac{t_i - t_{i-1} - \brac{v_i - v_{i-1}}}
 \end{align}
 for all $2 \le i \le n$ and $m\,|\, (t_1 - v_1)$. From these properties we can easily prove inductively that $m\,|\,\brac{t_i - v_i}$ for all $i \in [n]$. In particular $m\,|\,\brac{t_n - v_n}$ so we can define 
 \begin{align*}
     k = \frac{1}{m} \brac{t_n - v_n} \pmod{(n+1)} \in \mathbb{Z}.
 \end{align*}
 Then $0 \le k \le n$ and
\begin{align*}
\brac{n+1}\,\Big|\, \brac{\frac{1}{m} \brac{t_n - v_n} - k}.
\end{align*}
Therefore, from $(\ref{prop3})$ we can define 
 \begin{align*}
     &u_n = \frac{1}{n+1} \brac{\frac{1}{m} \brac{t_n - v_n} - k} \in \mathbb{Z}, \\&
     u_{i-1} = u_i - \frac{1}{im} \brac{t_i - t_{i-1} - \brac{v_i - v_{i-1}}} \in \mathbb{Z}
 \end{align*}
  for all $2\le i \le n$ so $u = (u_1, u_2, \ldots, u_n) \in \mathbb{Z}^n$. We have 
  \begin{align*}
      \brac{v + Au + k \mathbbm{1}_m}_n = v_n + m(n+1)u_n + mk = t_n
  \end{align*}
  and we can easily prove inductively that $\brac{v + Au + k \mathbbm{1}_m}_i = t_i$ for all $i \in [n]$ so $v + Au + k \mathbbm{1}_m = t$ and thus $t \in V_k$.

Since the family $\left\{V_k\right\}_{k=0}^n$ consists of pairwise disjoint sets and $\mathbb{Z}^n = \bigcup_{k=0}^n V_k$, then for every $t \in \mathbb{Z}^n$ there exists exactly one $0 \le k \le n$ such that $t \in V_k$. We define a function $f \colon \mathbb{Z}^n \to [n+1]$ as follows: $f(t) = k+1$ where $0 \le k \le n$ is the only one $k$ such that $t \in V_k$. We claim that $f$ is an $m$-distance clustered $(n+1)$-coloring of $G_n^\infty$. 

Indeed, let us fix $i \in [n+1]$ and notice that every connected component of the subgraph induced by the set $f^{-1}\bracl{\left\{i\right\}}$ coincides with a $1$-connected component of $f^{-1}\bracl{\left\{i\right\}}$. We have $f^{-1}\bracl{\left\{i\right\}} = V_{i-1} = \bigcup_{u \in \mathbb{Z}^n} V_{i-1}^u$. Set $\mathcal{V}$ is $1$-connected and $\abs{\mathcal{V}} = n!\,m^n$ so all $V_{i-1}^u$ are $1$-connected and $\abs{V_{i-1}^u} = \abs{\mathcal{V}} = n!\,m^n$ since they are some translations of $\mathcal{V}$. To complete the proof it is enough to show that $\text{dist} \brac{V_{i-1}^{u}, V_{i-1}^{u'}} > m$ for every distinct points $u, u' \in \mathbb{Z}^n$ since in that way we indicate that every $1$-connected component of $f^{-1}\bracl{\left\{i\right\}}$ coincides with $V_{i-1}^u$ for some $u \in \mathbb{Z}^n$ and the distance between any two of these components is bigger than $m$.

Let $u, u' \in \mathbb{Z}^n$ be such that $u \neq u'$ and let us define $\widehat{u} = u - u'$. Then we can take $i_0 = \max \left\{i \colon\, \widehat{u}_i \neq 0\right\}$. Let $v^u \in V_{i-1}^u$ and $v^{u'} \in V_{i-1}^{u'}$ be arbitrary points. Then $v^u = v + Au + (i-1)\mathbbm{1}_m$ and $v^{u'} = v' + Au' + (i-1)\mathbbm{1}_m$ for some $v, v' \in \mathcal{V}$. Finally, we estimate
\begin{align*}
    \norm{v^u - v^{u'}}_\infty &= \norm{v - v' + A\widehat{u}}_\infty \ge \abs{\brac{v - v' + A\widehat{u}}_{i_0}} \ge \abs{\brac{A\widehat{u}}_{i_0}} - \abs{\brac{v - v'}_{i_0}} \\ & = \abs{m(i_0+1)\widehat{u}_{i_0} + m\sum_{j=i_0+1}^n \widehat{u}_j} - \abs{v_{i_0} - v'_{i_0}} = \abs{m(i_0+1)\widehat{u}_{i_0}} - \abs{v_{i_0} - v'_{i_0}} \\ & \ge m(i_0 +1)\abs{\widehat{u}_{i_0}} - (mi_0 - 1) \ge m(i_0+1) - (mi_0-1) = m+1 > m.
\end{align*}
Hence $\text{dist} \brac{V_{i-1}^{u}, V_{i-1}^{u'}} > m$ and it completes the proof that $\chi_{\star, m}(G_n^\infty) \le n+1$ and $\widehat{K}_{\star, m}\brac{G_n^\infty} \le n!\,m^n$.
\end{proof}

\section{The main combinatorial result}\label{sectionmaincombi}

In this section, we present a proof of the main combinatorial result, Theorem \ref{maincombi}, stated in the Introduction. To this end, we combine Proposition \ref{graph} with Theorem \ref{chessboardthm}.

\begin{twmainA}
Let $n \in \mathbb{N}, m \in \mathbb{Z}_+$ be such that $0 \le m \le n-1$. Then, there exists a constant $C_{n, m}>0$ such that the following property holds:
\begin{itemize}
    \item[] Let $k \in \mathbb{N}$ and let $F \colon \mathcal{K}_k^n \to \mathbb{Z}^{n-1}$ be a function such that $\norm{F(K_1) - F(K_2)}_\infty \le 1$ if $\dim\brac{K_1 \cap K_2} \ge m$. Then there exist a $1$-connected subset $P \subset \mathbb{Z}^{n-1}$ with $\abs{P} \le C_{n, m}$ and a~subset $\mathcal{S} \subset F^{-1}\bracl{P}$ such that $\bigcup \mathcal{S}$ connects some opposite faces of $I^n$.
\end{itemize}
\end{twmainA}

\begin{proof}
    Let $n, m \in \mathbb{N}$ be such that $0 \le m \le n-1$. The case $n=1$ is trivial so we can assume that $n \ge 2$. Let $k \in \mathbb{N}$ and $F \colon \mathcal{K}_k^n \to \mathbb{Z}^{n-1}$ be a function such that $\norm{F(K_1) - F(K_2)}_\infty \le 1$ if $\dim\brac{K_1 \cap K_2} \ge m$.
    
    By virtue of Proposition \ref{graph} we have $\chi_{\star, m+1}\brac{G_{n-1}^\infty} = n$. Hence there exists a function \hbox{$f \colon \mathbb{Z}^{n-1} \to [n]$} such that for every $i \in [n]$ all $1$-connected components of the set $f^{-1}\bracl{\left\{i\right\}}$ have cardinalities at most $\widehat{K}_{\star, m+1}\brac{G_{n-1}^\infty}$ and the distance between any two of these $1$-connected components is bigger than $m+1$. 
    
    Let $\widehat{F} = f \circ F \colon \mathcal{K}_k^n \to [n]$. By virtue of Theorem \ref{chessboardthm} there exist $p \in [n]$ and $\mathcal{S} \subset \widehat{F}^{-1}\bracl{\left\{p\right\}}$ such that $\bigcup \mathcal{S}$ connects some opposite faces of $I^n$. Then $F\bracl{\mathcal{S}} \subset F  \bracl{\widehat{F}^{-1}\bracl{\left\{p\right\}}} = F \bracl{F^{-1} \bracl{f^{-1}\bracl{\left\{p\right\}}}} \subset f^{-1}\bracl{\left\{p\right\}}$. We claim that $F\bracl{\mathcal{S}}$ is contained in some $1$-connected component of $f^{-1}\bracl{\left\{p\right\}}$.

    Indeed, let $K, K' \in \mathcal{S}$. Since $\bigcup \mathcal{S}$ is connected, then by Observation \ref{obsforj} $(1)$ there exists a sequence $c \colon [L] \to \mathcal{S}$ for some $L>0$ such that $c_1 = K, c_L = K'$ and $c_i \cap c_{i+1} \neq \emptyset$ for $i \in [L-1]$. Since $\norm{F(K_1) - F(K_2)}_\infty \le 1$ if $\dim \brac{K_1 \cap K_2} \ge m$, then $\norm{F(K_1) - F(K_2)}_\infty \le m+1$ if $\dim \brac{K_1 \cap K_2} \ge 0$ by Observation \ref{obsforj} $(6)$. Thus $\norm{F(c_i) - F(c_{i+1})}_\infty \le m+1$ for $i \in [L-1]$. It means that all points $F(c_i)$ belong to the same $1$-connected component of $f^{-1}\bracl{\left\{p\right\}}$ since the distance between any two $1$-connected components of $f^{-1}\bracl{\left\{p\right\}}$ is bigger than $m+1$. In particular, $F(K)$ and $F(K')$ belong to the same $1$-connected component of $f^{-1}\bracl{\left\{p\right\}}$.

    Let $P$ be the $1$-connected component of $f^{-1}\bracl{\left\{p\right\}}$ such that $F\bracl{\mathcal{S}} \subset P$. Then $\mathcal{S} \subset F^{-1}\bracl{P}$ and $\abs{P} \le \widehat{K}_{\star, m+1}\brac{G_{n-1}^\infty}$ so it completes the proof.
\end{proof}

For $0 \le m \le n-1$ by $\widehat{C}_{n, m}$ we denote the minimum constant $C_{n, m}>0$ for which the property from Theorem \ref{maincombi} holds. Obviously, $\widehat{C}_{n, m_1} \le \widehat{C}_{n, m_2}$ for $m_1 \le m_2$. We shall examine some minimal constants in the cases $n = 2$ and $n = 3$ in Section \ref{minimalconst}. The degenerated case $n=1$ is trivial and $\widehat{C}_{1, 0} = 1$. Moreover we will show in Proposition \ref{optimalconstprop} that for $n \ge 3$ and every $0 \le m \le n-1$ we have $\widehat{C}_{n,m} >1$.

\begin{cor}\label{ineqconst}
    For all $0 \le m \le n-1$ we have $\widehat{C}_{n, m} \le \widehat{K}_{\star, m+1}\brac{G_{n-1}^\infty} \le (n-1)!\, (m+1)^{n-1}$. \qed
\end{cor}

\section{Minimum constants}\label{minimalconst}

In this section, we investigate the optimal constants appearing in Theorem \ref{maincombi}. More precisely, we treat the two-dimensional case by showing that $\widehat{C}_{2, 0} = \widehat{C}_{2, 1} = 1$ (Proposition \ref{optimalconst1}), we show that in the $n$-dimensional case, with $n\ge 3$, the optimal constants are strictly greater than $1$ (Proposition \ref{optimalconstprop}), and we establish that $\widehat{C}_{3, 0} = 2$ (Corollary \ref{optimalconst2}).

\begin{prop}\label{optimalconst1}
$\widehat{C}_{2, 0} = \widehat{C}_{2, 1} = 1$.
\end{prop}

\begin{proof}
From Corollary \ref{ineqconst} we obtain $\widehat{C}_{2, 0} \le 1$ so we only have to show that $\widehat{C}_{2, 1} \le 1$.

Let $k \in \mathbb{N}$ and let $F \colon \mathcal{K}_k^2 \to \mathbb{Z}$ be a function such that $\norm{F(K_1) - F(K_2)}_\infty \le 1$ if $\dim\brac{K_1 \cap K_2} \ge 1$. Let $\mathcal{A}$ be the family of all maximal (under inclusion) sets $\mathcal{S} \subset F^{-1}\left[\left\{p\right\}\right]$ such that $p \in \mathbb{Z}$ and $\bigcup \mathcal{S}$ is connected. It is enough to prove that there exists $\mathcal{S} \in \mathcal{A}$ such that $\bigcup \mathcal{S}$ connects some opposite faces of $I^2$. Conversely, let us suppose that for every $\mathcal{S} \in \mathcal{A}$ the set $\bigcup \mathcal{S}$ does not connect $i$th opposite faces of $I^2$ for any $i \in [2]$.

For $i \in [2]$ let $\pi_i \colon \mathbb{R}^2 \to \mathbb{R}$ be the projection onto the $i$th coordinate. Let $\mathcal{S}_{\text{max}}$ be an arbitrary $\mathcal{S} \in \mathcal{A}$ which maximizes value $\max \,\pi_1\bracl{\bigcup \mathcal{S}} - \min \, \pi_1\bracl{\bigcup \mathcal{S}}$ over all $\mathcal{S} \in \mathcal{A}$ and let $p \in \mathbb{Z}$ be such that $\mathcal{S}_\text{max} \subset F^{-1}\bracl{\left\{p\right\}}$. For fixed $i \in [2]$ since $\bigcup \mathcal{S}_\text{max}$ does not connect $i$th opposite faces of $I^2$, it must be $\max \pi_i\bracl{\bigcup \mathcal{S_{\text{max}}}} < 1$ or $\min \pi_i\bracl{\bigcup \mathcal{S_{\text{max}}}} > 0$. Without loss of generality\footnote{There are four cases to consider, but they are all similar.} we can assume that $\min \pi_i\bracl{\bigcup \mathcal{S_{\text{max}}}} > 0$ for both $i \in [2]$. Let $x_0 = \brac{1/2, 1\,/\brac{2k}}$. Then $x_0 \in \interior{I^2} \setminus \bigcup \mathcal{S}_{\text{max}}$ since $\min \pi_2\bracl{\bigcup \mathcal{S_{\text{max}}}} > 0$. Let $S$ be the connected component of $\interior{I^2} \setminus \bigcup \mathcal{S}_{\text{max}}$ containing the point $x_0$. Finally, let 
%\vspace{5mm}
\begin{align*}
&\mathcal{S}_0=\left\{K \in \mathcal{K}_k^2\colon\, K \cap S \neq \emptyset\right\},\\
&\mathcal{S}_1 = \left\{K \in \mathcal{S}_0\colon\, K \cap \partial_{\interior{I^2}} S \neq \emptyset\right\},\\
&\mathcal{S}_2 = \left\{K \in \mathcal{S}_0\colon\, \dim\brac{K \cap \partial_{\interior{I^2}} S} = 1 \right\}.
\end{align*}

\begin{figure}[h!]
\centering
     \includegraphics[width=0.86\linewidth]{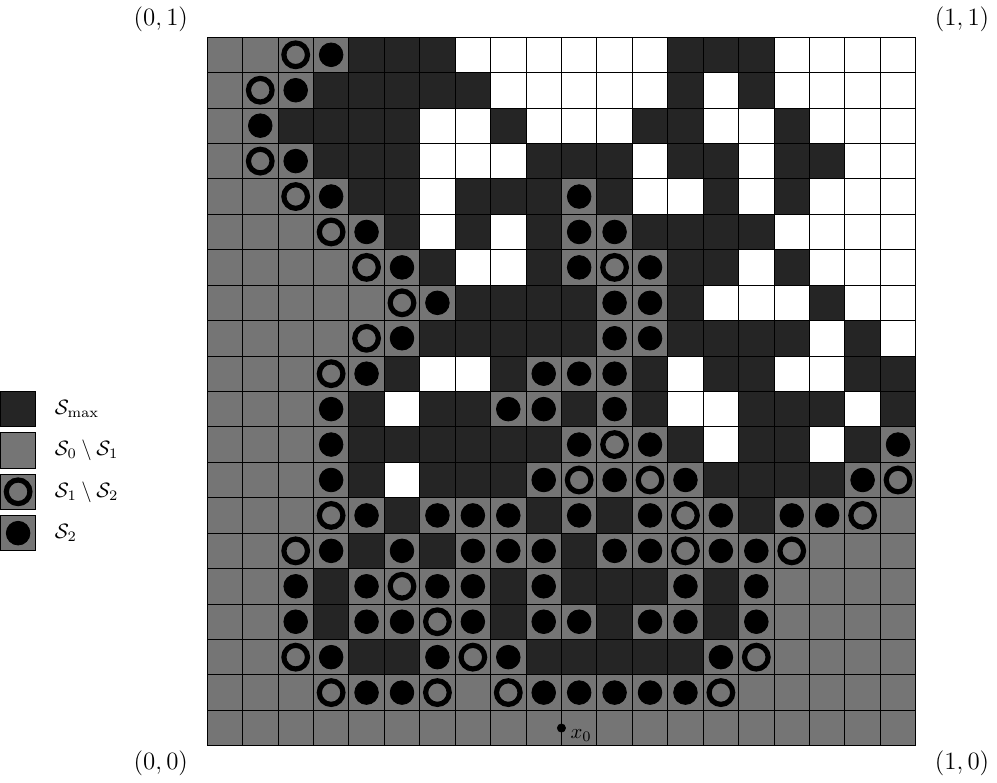}
     \caption{An illustration of a generic situation from the proof.}
     \label{figchess}
\end{figure}
\smallskip

We shall show the following properties\footnote{We provide rigorous arguments for all properties but once the drawing (e.g. Figure \ref{figchess}) is made properties $(\ref{dow1}) - (\ref{dow5})$ and $(\ref{dow8})$ should be visually intuitive.} of $\mathcal{S}_1$, $\mathcal{S}_2$ and $S$:
\begin{enumerate}
\item the set $\partial_{\interior{I^2}} S$ is connected and $\partial_{\interior{I^2}} S \subset \bigcup \mathcal{S}_{\text{max}}$; \label{dow1}
\item $\partial_{\interior{I^2}} S \subset \bigcup \mathcal{S}_2$; \label{dow3}
\item the set $\bigcup \mathcal{S}_2$ is connected; \label{dow4}
\item for every distinct $K, K' \in \mathcal{S}_1$ there exists a sequence $s \colon [N] \to \mathcal{S}_1$ for some $N>0$ such that $s_1 = K, s_N = K'$ and $\dim\brac{s_i \cap s_{i+1}} = 1$ for all $i \in [N-1]$; \label{dow5}
\item $\mathcal{S}_1 \subset F^{-1}\bracl{\left\{p-2, p-1\right\}}$ or $\mathcal{S}_1 \subset F^{-1}\bracl{\left\{p+1, p+2\right\}}$; \label{dow6}
\item $\mathcal{S}_2 \subset F^{-1}\bracl{\left\{p-1\right\}}$ or $\mathcal{S}_2 \subset F^{-1}\bracl{\left\{p+1\right\}}$; \label{dow7}
\item $\max \,\pi_1\bracl{\bigcup \mathcal{S}_{2}} - \min \, \pi_1\bracl{\bigcup \mathcal{S}_{2}} > \max \,\pi_1\bracl{\bigcup \mathcal{S}_{\max}} - \min \, \pi_1\bracl{\bigcup \mathcal{S}_{\max}}$. \label{dow8}
\end{enumerate}
Once we show properties $(\ref{dow4})$, $(\ref{dow7})$ and $(\ref{dow8})$, the extension of the set $\mathcal{S}_2$ to a set belonging to family $\mathcal{A}$ indicates a contradiction with the definition of $\mathcal{S}_{\max}$ which completes the proof.
\medskip

$(\ref{dow1})$: It follows from Proposition \ref{toconstC21} (i) and (ii) since $\interior{I^2}$ is homeomorphic with $\mathbb{R}^2$ and $\bigcup \mathcal{S}_{\text{max}}$ is closed and connected.
\medskip

$(\ref{dow3})$: Let $x \in \partial_{\interior{I^2}} S$ and let $K_1, K_2 \in \mathcal{K}_k^2$ be distinct squares such that $x \in K_1 \cap K_2$ and $\dim\brac{K_1 \cap K_2 \cap \partial_{\interior{I^2}} S} = 1$. Then, there exists a point $y \in \partial_{\interior{I^2}} S$ such that $B\brac{y, 1\,/\brac{2k}} \subset K_1 \cup K_2$. From the definition of $\partial_{\interior{I^2}} S$ we have $B\brac{y, 1\,/\brac{2k}} \cap S \neq \emptyset$ so $S \cap \brac{K_1 \cup K_2} \neq \emptyset$. Without loss of generality we can assume $S \cap K_1 \neq \emptyset$ so $K_1 \in \mathcal{S}_0$. Moreover $\dim\brac{K_1 \cap K_2 \cap \partial_{\interior{I^2}} S} = 1$ so $\dim\brac{K_1 \cap \partial_{\interior{I^2}} S} = 1$ and thus $K_1 \in \mathcal{S}_2$. Since $x \in K_1$, then $x \in \bigcup \mathcal{S}_2$.
\medskip

$(\ref{dow4})$: Taking into account the definition of $\mathcal{S}_2$ and properties $(\ref{dow1})$ and $(\ref{dow3})$ the result follows immediately from Lemma \ref{conn}.
\medskip

$(\ref{dow5})$: Let $V$ be the set of vertices of all squares from $\mathcal{K}_k^2$ and let $\widetilde{B} = S \cap \bigcup_{x \in \interior{I^2} \setminus S} B_{l^\infty}\brac{x, 1/k}$. By virtue of Proposition \ref{toconstC21} (iii) the set $\widetilde{B}$ is open (in $\interior{I^2}$ so also in $\mathbb{R}^2$) and connected. Thus $B := \widetilde{B} \setminus V$ is also open and connected since $V$ is finite, so $B$ is path-connected. We shall show that for every $K \in \mathcal{K}_k^2$ we have $K \in \mathcal{S}_1 \iff K \cap B \neq \emptyset$.

Indeed, let us fix $K \in \mathcal{S}_1$ and take $y \in \interior{K}$. We have $K \cap S \neq \emptyset$ so obviously $\interior{K} \subset S$ and then $y \in S$. Moreover $K \cap \partial_{\interior{I^2}} S \neq \emptyset$ and $\partial_{\interior{I^2}} S \subset \bigcup \mathcal{S}_{\max}$ from $(\ref{dow1})$ so there exists $K' \in \mathcal{S}_{\max}$ such that $K \cap K' \neq \emptyset$. Hence $\norm{y - x}_\infty < 1/k$ for some $x \in K'$. But $K' \subset \bigcup \mathcal{S}_{\max} \subset \interior{I^2} \setminus S$ so $x \in \interior{I^2} \setminus S$. In this way we showed that $\interior{K} \subset B$ so in particular $K \cap B \neq \emptyset$.

Now, let us suppose that $K \cap B \neq \emptyset$. Obviously $K \cap S \neq \emptyset$ so $K \in \mathcal{S}_{0}$. We take $y \in K \cap B$ and $x \in \interior{I^2} \setminus S$ such that $\norm{y - x}_\infty < 1/k$. Let us notice that for every $K \in \mathcal{K}_k^2$ either $\interior{K} \subset S$ or $\interior{K} \subset \interior{I^2} \setminus S$. Using this observation we can conclude that there exists $K' \in \mathcal{K}_k^2$ such that $x \in K'$ and $\interior{K'} \subset \interior{I^2} \setminus S$. Since $\norm{y - x}_\infty < 1/k$ and $x \in K'$ we get $K' \cap K \neq \emptyset$. But $\interior{K} \subset S$ and $\interior{K'} \subset \interior{I^2} \setminus S$ so $K \cap K' \cap \partial_{\interior{I^2}} S \neq \emptyset$. Thus $K \cap \partial_{\interior{I^2}} S \neq \emptyset$ so finally $K \in \mathcal{S}_1$.

We are now ready to define the required sequence $s$. Let $K, K' \in \mathcal{S}_1$ be distinct. Then $K \cap B \neq \emptyset \neq K' \cap B$. Let $\mathcal{B} = \left\{K \in \mathcal{K}_k^2 \colon\, K \cap B \neq \emptyset\right\}$. Since the set $B$ is path-connected and $B \cap V = \emptyset$, then there exists a sequence $s \colon [N] \to \mathcal{B}$ for some $N>0$ such that $s_1 = K$, $s_N = K'$ and $\dim \brac{s_i \cap s_{i+1}} = 1$ for all $i \in [N-1]$. But $\mathcal{B} \subset \mathcal{S}_1$ so it completes the proof of this property.
\medskip

$(\ref{dow6})$: Let $K \in \mathcal{S}_1$. Then $K \cap \partial_{\interior{I^2}} S \neq \emptyset$ so $K \cap \bigcup \mathcal{S}_{\max} \neq \emptyset$ from property $(\ref{dow1})$. Since $F\left[\mathcal{S}_{\max}\right] = \left\{p\right\}$, then $\abs{F(K) - p} \le 2$ from the property of the function $F$. Moreover $0 < \abs{F(K) - p} \le 2$ from the maximality of $\mathcal{S}_{\max}$ so $\mathcal{S}_1 \subset F^{-1}\left[\left\{p-2, p-1, p+1, p+2\right\}\right]$.

Let $K_1, K_2 \in \mathcal{S}_1$ be distinct squares and let us suppose that $K_1 \in F^{-1}\left[\left\{p+1, p+2\right\}\right]$. We shall show that $K_2 \in F^{-1}\left[\left\{p+1, p+2\right\}\right]$. Having in mind property $(\ref{dow5})$ we can assume that \hbox{$\dim \brac{K_1 \cap K_2} = 1$}. Then, $\abs{F(K_1) - F(K_2)} \le 1$ so $F(K_2) \in \left\{p+1, p+2\right\}$. Thus we obtain $\mathcal{S}_1 \subset F^{-1}\left[\left\{p+1, p+2\right\}\right]$. Analogously we get $\mathcal{S}_1 \subset F^{-1}\left[\left\{p-2, p-1\right\}\right]$ if we assume that $K_1 \in F^{-1}\left[\left\{p-2, p-1\right\}\right]$.
\medskip

$(\ref{dow7})$: From the definition of $\mathcal{S}_2$ and property $(\ref{dow1})$ we can conclude that for every $K \in \mathcal{S}_2$ there exists $K' \in \mathcal{S}_{\max}$ such that $\dim\brac{K \cap K'} = 1$. Then, from the property of the function $F$ and the maximality of $\mathcal{S}_{\max}$ we get $\mathcal{S}_2 \subset F^{-1}\left[\left\{p-1, p+1\right\}\right]$. But $\mathcal{S}_2 \subset \mathcal{S}_1$ so $\mathcal{S}_2 \subset F^{-1}\left[\left\{p-2, p-1\right\}\right]$ or $\mathcal{S}_2 \subset F^{-1}\left[\left\{p+1, p+2\right\}\right]$ from property $(\ref{dow6})$. Hence $\mathcal{S}_2 \subset F^{-1}\left[\left\{p-1\right\}\right]$ or $\mathcal{S}_2 \subset F^{-1}\left[\left\{p+1\right\}\right]$.
\vspace{10pt}

$(\ref{dow8})$: Firstly we shall show that $\max \,\pi_1\bracl{\bigcup \mathcal{S}_{2}} \ge \max \,\pi_1\bracl{\bigcup \mathcal{S}_{\max}}$. For this purpose, among all squares $K \in \mathcal{S}_{\max}$ such that $\max \pi_1\left[K\right] = \max \pi_1\left[\bigcup \mathcal{S}_{\max}\right]$ let us choose $K_0$ for which $\min \pi_2\left[K_0\right]$ has the least value. Obviously $\min \pi_2\left[K_0\right]>0$ since $\min \pi_2\left[\bigcup \mathcal{S}_{\max}\right] > 0$, so there exists $K' \in \mathcal{K}_{k}^2$ such that $\dim\brac{K_0 \cap K'} = 1$ and $\min \pi_2\left[K'\right] < \min \pi_2\left[K_0\right]$ i.e. $K'$ lies immediately below $K_0$. Then $\max \pi_1\left[K'\right] = \max \pi_1\left[K_0\right]$. Let us recall that $S$ is the connected component of $\interior{I^2} \setminus \bigcup \mathcal{S}_{\text{max}}$ containing the point $x_0 = \brac{1/2, 1\,/\brac{2k}}$. Let $K'' \in \mathcal{K}_k^2$ be the square such that $\min \pi_2\bracl{K''} = 0$ and $\max \pi_1\left[K''\right] = \max \pi_1\left[K_0\right]$. Let $x''$ be the center of $K''$, $x'$ be the center of $K'$, $L_1$ be the line segment connecting the point $x_0$ with $x''$ and $L_2$ be the line segment connecting the point $x''$ with $x'$. Then $L := L_1 \cup L_2$ is connected and $L \subset \interior{I^2} \setminus \bigcup \mathcal{S}_{\max}$ from the definitions of $K''$ and $K'$, and the fact that $\min \pi_2\left[\bigcup \mathcal{S}_{\max}\right] > 0$. Moreover $x_0, x' \in L$ so $x' \in S$ and then $K' \cap S \neq \emptyset$. Since $K_0 \in \mathcal{S}_{\max}$, $\dim\brac{K_0 \cap K'} =1$ and $K' \cap S \neq \emptyset$, we obtain $\dim\brac{K' \cap \partial_{\interior{I^2}} S} = 1$ and $K' \in \mathcal{S}_0$ so $K' \in \mathcal{S}_2$. Finally, we conclude that $\max \,\pi_1\bracl{\bigcup \mathcal{S}_{2}} \ge \max \,\pi_1\bracl{K'} = \max \,\pi_1\bracl{\bigcup \mathcal{S}_{\max}}$.

We show the inequality $\min \,\pi_1\bracl{\bigcup \mathcal{S}_{2}} < \min \,\pi_1\bracl{\bigcup \mathcal{S}_{\max}}$ in the similar manner. Among all squares $K \in \mathcal{S}_{\max}$ such that $\min \pi_1\left[K\right] = \min \pi_1\left[\bigcup \mathcal{S}_{\max}\right]$ let us choose $K_0$ for which $\min \pi_2\left[K_0\right]$ has the least value (which is bigger than zero). There exists $K' \in \mathcal{K}_{k}^2$ which lies immediately below $K_0$ and we can show that $K' \in \mathcal{S}_2$ as before. Since $\min \pi_1\bracl{\bigcup \mathcal{S}_{\max}} > 0$, we get $\min \pi_1\bracl{K_0} > 0$. Let $K_1, K_2 \in \mathcal{K}_k^2$ be the squares such that $K_1$ lies immediately to the left of $K_0$ and $K_2$ lies immediately below $K_1$. From the definition of $K_0$ we have $K_1, K_2 \notin \mathcal{S}_{\max}$. Thus $K_1 \cap S \neq \emptyset \neq K_2 \cap S$ since $S$ is the connected component of $\interior{I^2} \setminus \bigcup \mathcal{S}_{\max}$ and $K' \cap S \neq \emptyset$. Moreover $K_0 \in \mathcal{S}_{\max}$ and \hbox{$\dim\brac{K_0 \cap K_1} = 1$} so $\dim\brac{K_1 \cap \partial_{\interior{I^2}} S} = 1$, and then $K_1 \in \mathcal{S}_2$. Therefore we can conclude that $\min \,\pi_1\bracl{\bigcup \mathcal{S}_{2}} \le \min \,\pi_1\bracl{K_1} = \min \,\pi_1\bracl{K_0} - 1/k = \min \,\pi_1\bracl{\bigcup \mathcal{S}_{\max}} - 1/k$.
\end{proof}

\begin{rem}
Let us notice that $\widehat{K}_{\star, 2}\brac{G_{1}^\infty} \le 2$ from Proposition \ref{graph} and it is easy to see that $\widehat{K}_{\star, 2}\brac{G_{1}^\infty} > 1$ so $\widehat{K}_{\star, 2}\brac{G_{1}^\infty} = 2$. On the other hand, we have $\widehat{C}_{2, 1} = 1$ from Proposition \ref{optimalconst1} so in general, constants $\widehat{C}_{n, m}$ and $\widehat{K}_{\star, m+1}\brac{G_{n-1}^\infty}$ may not be equal (cf. Corollary \ref{ineqconst}).
\end{rem}

\begin{prop}\label{optimalconstprop}
For all $0 \le m \le n-1$ such that $n \ge 3$ we have $\widehat{C}_{n, m} >1$.
\end{prop}

\begin{proof}
Let $0 \le m \le n-1$ and $n \ge 3$. Conversely, let us suppose that $\widehat{C}_{n, m} = 1$. Since $\chi_{\star, 1} \brac{G_n^\infty} = n+1$, there exists a function $f \colon \mathbb{Z}^n \to [n+1]$ such that for every $i \in [n+1]$ all $1$-connected components of $f^{-1}\left[\left\{i\right\}\right]$ have cardinalities at most $\widehat{K}_{\star, 1} \brac{G_n^\infty}$. Let $M = \widehat{K}_{\star, 1} \brac{G_n^\infty} + 1$. Since $n \ge 3$, there exists a one-to-one function $g \colon [n+1] \to \left\{0, 1\right\}^{n-1}$ so $g$ is a bijection on its image. Let us define \hbox{$F = g \circ f \circ j_{n, M}^{-1} \colon \mathcal{K}_{M}^n \to \left\{0, 1\right\}^{n-1} \subset \mathbb{Z}^{n-1}$}.

Obviously for every $K_1, K_2 \in \mathcal{K}_M^n$ we have $\norm{F(K_1) - F(K_2)}_\infty \le 1$. Since we supposed $\widehat{C}_{n, m} = 1$, there exist $p \in \mathbb{Z}^{n-1}$ and $\mathcal{S} \subset F^{-1}\left[\left\{p\right\}\right]$ such that $\bigcup \mathcal{S}$ connects some opposite faces of $I^n$. We argue in the same way as in Proposition \ref{graph} that $j_{n, M}^{-1}\bracl{\mathcal{S}}$ is $1$-connected and $\abs{j_{n, M}^{-1}\bracl{\mathcal{S}}} > \widehat{K}_{\star, 1}\brac{G_n^\infty}$. Moreover $g \circ f \circ j_{n,M}^{-1}\bracl{\mathcal{S}} = F\bracl{\mathcal{S}} = \left\{p\right\}$ so $j_{n,M}^{-1}\bracl{\mathcal{S}} \subset f^{-1}\left[\left\{g^{-1}(p)\right\}\right]$ and it contradicts the fact that all $1$-connected components of $f^{-1}\left[\left\{g^{-1}(p)\right\}\right]$ have cardinalities at most $\widehat{K}_{\star, 1}\brac{G_n^\infty}$.
\end{proof}

\begin{cor}\label{optimalconst2}
    $\widehat{C}_{3, 0} = 2$.
\end{cor}

\begin{proof}
On the one hand $\widehat{C}_{3, 0} > 1$ from Proposition \ref{optimalconstprop} but on the other hand $\widehat{C}_{3, 0} \le 2$ from Corollary \ref{ineqconst}.
\end{proof}

\section{The continuous result}\label{sectionmainconti}

In this section, we prove Theorem \ref{contthm}, stated in the Introduction, using Theorem \ref{maincombi} together with techniques based on Hausdorff convergence. We can interpret Theorem \ref{contthm} as a continuous analogue of Theorem \ref{maincombi}. Moreover, we show that Theorem \ref{contthm} cannot be generalized to path-connected sets.

\begin{twmainB}
Let $n \in \mathbb{N}$ and $f \colon I^n \to \mathbb{R}^{n-1}$ be a continuous function. Then there exist a point $p \in \mathbb{R}^{n-1}$ and a compact subset $S \subset f^{-1}\bracl{\left\{p\right\}}$ which connects some opposite faces of $I^n$.
\end{twmainB}

\begin{proof}
We will first consider the approximate problem.

\begin{lem}\label{contlemma}
Let $n \in \mathbb{N}$ and $f \colon I^n \to \mathbb{R}^{n-1}$ be a continuous function. Then for every $\varepsilon > 0$ there exist a point $p \in \mathbb{R}^{n-1}$ and a compact subset $S \subset f^{-1}\bracl{B(p, \varepsilon)}$ which connects some opposite faces of $I^n$.
\end{lem}

\begin{proof}
The case $n=1$ is trivial so we assume $n \ge 2$. Let us fix $\varepsilon>0$. Obviously $f\bracl{I^n}$ is compact so we can assume\footnote{If $f\bracl{I^n} \subset \bracl{-M, M}^{n-1}$ for some $M>0$ and $x_0 = (1/2, \ldots, 1/2) \in \mathbb{R}^{n-1}$, then for $f_0 = f\,/\brac{2M}+ x_0$ we have $f_0\bracl{I^n} \subset I^{n-1}$. If $S \subset f_0^{-1}\bracl{B\brac{p, \varepsilon}}$, then $S \subset f^{-1}\bracl{B\brac{2M(p - x_0), 2M\varepsilon}}$.} (after possible rescaling and translation of $f$) that $f\bracl{I^n} \subset I^{n-1}$. Let $k \in \mathbb{N}$ be such that $3 \sqrt{n-1}\,\widehat{C}_{n, 0}\,/\brac{2k} < \varepsilon$. Then, the family $\mathcal{K}_k^{n-1}$ covers $I^{n-1}$ and $\diam K = \sqrt{n-1}/k$ for every $K \in \mathcal{K}_k^{n-1}$. We slightly enlarge cubes from the family $\mathcal{K}_k^{n-1}$ to open cubes. Namely, let $h \colon \mathcal{K}_{k}^{n-1} \to \mathcal{K}'$ where $\mathcal{K}' = h\bracl{\mathcal{K}_{k}^{n-1}}$ be a bijection defined as $h\brac{K} = \bigcup_{x \in K} B_{l^\infty} \brac{x, 1\,/\brac{4k}}$ for $K \in \mathcal{K}_k^{n-1}$. Obviously, for $K' \in \mathcal{K}'$ we have $\diam K' = \brac{1/k \;+\; 1\,/\brac{2k}} \sqrt{n-1} = 3\sqrt{n-1}\,/\brac{2k}$. Moreover, the following property of the function $h$ holds: for every $K_1, K_2 \in \mathcal{K}_k^{n-1}$ we have $K_1 \cap K_2 \neq \emptyset$ if and only if $h\brac{K_1} \cap h\brac{K_2} \neq \emptyset$.

The family $\mathcal{U} = \left\{f^{-1}\bracl{K'} \colon K' \in \mathcal{K}'\right\}$ is an open cover of $I^n$. Let $\delta >0$ be the Lebesgue number of this cover i.e. if $A \subset I^n$ and $\diam A < \delta$, then $A \subset f^{-1}\bracl{K'}$ for some $K' \in \mathcal{K}'$. Let $m \in \mathbb{N}$ be such that $\sqrt{n}/m < \delta$. Thus, if $K \in \mathcal{K}_m^n$, then $\diam K = \sqrt{n}/m < \delta$ so $K \subset f^{-1}\bracl{K'}$ for some $K' \in \mathcal{K}'$. Hence we can choose some function $g \colon \mathcal{K}_m^n \to \mathcal{K'}$ such that $K \subset f^{-1} \bracl{g\brac{K}}$ for $K \in \mathcal{K}_m^n$. Finally we define $F \colon \mathcal{K}_m^n \to [k]^{n-1}$ as $F = \widehat{j}_{n-1, k}^{-1} \circ g$ where $\widehat{j}_{n-1, k} = h \circ j_{n-1, k}$. Hence if $F(K) = i$, then $f\bracl{K} \subset \widehat{j}_{n-1, k}(i)$.

We shall show that for every $K_1, K_2 \in \mathcal{K}_m^n$ we have $\norm{F(K_1) - F(K_2)}_\infty \le 1$ if $\dim \brac{K_1 \cap K_2} \ge 0$. Let us fix $K_1, K_2 \in \mathcal{K}_m^n$ and suppose that $\dim \brac{K_1 \cap K_2} \ge 0$. Then $K_1 \cap K_2 \neq \emptyset$ so let us take $x \in K_1 \cap K_2$. Let $F(K_1) = i$ and $F(K_2) = i'$. Then $f\bracl{K_1} \subset \widehat{j}_{n-1, k}(i)$ and $f\bracl{K_2} \subset \widehat{j}_{n-1, k}(i')$ so $f(x) \in \widehat{j}_{n-1, k}(i) \cap \widehat{j}_{n-1, k}(i')$. In particular $\widehat{j}_{n-1, k}(i) \cap \widehat{j}_{n-1, k}(i') \neq\emptyset$ so $j_{n-1, k}(i) \cap j_{n-1, k}(i') \neq\emptyset$ by the property of the function $h$ and then $\norm{i - i'}_\infty \le 1$ by Observation \ref{obsforj} $(2)$.

By virtue of Theorem \ref{maincombi} there exist a $1$-connected subset $P \subset [k]^{n-1}$ with $\abs{P} \le \widehat{C}_{n, 0}$ and a~subset $\mathcal{S} \subset F^{-1}\bracl{P}$ such that $S := \bigcup \mathcal{S}$ connects some opposite faces of $I^n$. Let us notice that for $K \in F^{-1}\bracl{P}$ we have $K \subset f^{-1}\bracl{\bigcup \widehat{j}_{n-1, k}\bracl{P}}$ so $S \subset \bigcup F^{-1}\bracl{P} \subset f^{-1}\bracl{\bigcup \widehat{j}_{n-1, k}\bracl{P}}$. Let us take $p \in P$ and $z \in \widehat{j}_{n-1, k}(p)$. To complete the proof it is enough to show that $\bigcup \widehat{j}_{n-1, k}\bracl{P} \subset B\brac{z, \varepsilon}$.

For this purpose let $z' \in \bigcup \widehat{j}_{n-1, k}\bracl{P}$. Then $z' \in \widehat{j}_{n-1, k}(p')$ for some $p' \in P$. Since $P$ is $1$-connected and $\abs{P} \le \widehat{C}_{n, 0}$, there exists a sequence $s \colon [N] \to P$ for some $N>0$ such that $N \le \widehat{C}_{n, 0}$, $s_1 = p$, $s_N = p'$ and $\norm{s_{l} - s_{l+1}}_\infty \le 1$ for all $l \in [N-1]$. Thus $j_{n-1, k}(s_{l+1}) \cap j_{n-1, k}(s_l) \neq\emptyset$ for all $l \in [N-1]$ by Observation \ref{obsforj} $(2)$ and then $\widehat{j}_{n-1, k}(s_{l+1}) \cap \widehat{j}_{n-1, k}(s_l) \neq\emptyset$ for all $l \in [N-1]$. Let $x_0 = z, x_{N} = z'$ and $x_l \in \widehat{j}_{n-1, k}(s_{l}) \cap \widehat{j}_{n-1, k}(s_{l+1})$ for $l \in [N-1]$. Then $x_l, x_{l+1} \in \widehat{j}_{n-1, k}(s_{l+1})$ so $\norm{x_l - x_{l+1}} \le \diam \widehat{j}_{n-1, k}(s_{l+1}) = 3\sqrt{n-1}\,/\brac{2k}$ for all $0 \le l < N$. Finally,
\begin{align*}
    \norm{z - z'} = \norm{x_0 - x_N} \le \sum_{l=0}^{N-1} \norm{x_l - x_{l+1}} \le \sum_{l=0}^{N-1} \frac{3\sqrt{n-1}}{2k} = \frac{3\sqrt{n-1}}{2k} N \le \frac{3\sqrt{n-1}}{2k} \widehat{C}_{n, 0} < \varepsilon 
\end{align*}
so $z' \in B(z, \varepsilon)$.
\end{proof}

Now we are in position to prove Theorem \ref{contthm}. By virtue of Lemma \ref{contlemma} for every $m \in \mathbb{N}$ there exist a~compact set $S_m \subset I^n$ which connects $i_m$th opposite faces of $I^n$ for some $i_m \in [n]$ and a point $p_m \in \mathbb{R}^{n-1}$ such that $S_m \subset f^{-1}\bracl{\,\overline{B}\brac{p_m, 1/m}}$. There exists $i \in [n]$ such that infinitely many sets $S_m$ connect $i$th opposite faces of $I^n$ so we can assume that $i_m = i$ for every $m \in \mathbb{N}$. Obviously the sequence of points $p_m$ is bounded so it has a subsequence convergent to some $p \in \mathbb{R}^{n-1}$. We can assume that $\lim_{m\to\infty} p_m = p$. Since $S_m$ are compact sets and they are contained in the compact set $I^n$, by virtue of Theorem \ref{blaschke} $(2)$ the sequence of sets $S_m$ has a subsequence convergent in Hausdorff sense to some compact set $S$. We can assume that $\lim_{m\to\infty} S_m = S$. We shall show that $S$ connects $i$th opposite faces of $I^n$.

Indeed, sets $S_m$ are connected so $S$ is also connected by Theorem \ref{golab}. Sets $S_m$ connect $i$th opposite faces of $I^n$ so let us take $x_m \in S_m \cap I_{i, -}^n$. Since $x_m \in S_m$ and $\lim_{m\to\infty} S_m = S$, there exists a~subsequence~$x_{m_k}$ convergent to some $x \in S$ from Observation \ref{hausdorff}. On the other hand, $x_{m_k} \in I_{i, -}^n$ and the set $I_{i, -}^n$ is closed so $x \in I_{i, -}^n$. Hence $x \in S \cap I_{i, -}^n$ so $S \cap I_{i, -}^n \neq \emptyset$. We can show $S \cap I_{i, +}^n \neq \emptyset$ in the same way.

To complete the proof it is enough to show that $S \subset f^{-1}\bracl{\left\{p\right\}}$. For that, let us take $x \in S$. Then, from Observation \ref{hausdorff} there exists a sequence of points $x_m \in S_m \subset f^{-1}\bracl{\,\overline{B}\brac{p_m, 1/m}}$ convergent to $x$. Thus $\norm{f(x_m) - p_m} \le 1/m$ for every $m \in \mathbb{N}$ so $f(x) = p$ since $\lim_{m\to\infty} p_m = p$ and $\lim_{m\to\infty} f(x_m) = f(x)$. Hence $x \in f^{-1}\bracl{\left\{p\right\}}$.
\end{proof}

The following example shows that Theorem \ref{contthm} cannot be generalized to path-connected sets for any $n \ge 2$ (cf. \cite[Remarks]{waksman} for $n=2$).

\begin{ex}\label{counterexamplesinus}
Let $n \ge 2$. There exists a continuous (and even smooth) function $f \colon I^n \to \mathbb{R}^{n-1}$ such that for every $p \in \mathbb{R}^{n-1}$ every path-connected subset $S \subset f^{-1}\bracl{\left\{p\right\}}$ does not connect $i$th opposite faces of $I^n$ for any $i \in [n]$.
\end{ex}

\begin{proof}
Let us define a continuous function $g \colon I \setminus \left\{1/4\right\} \to I$ as follows:
\begin{align*}
g(x) = \begin{cases}
-2x+1 & \text{if } 0 \le x < \frac{1}{4}; \\
\frac{1}{4} \sin\brac{\frac{1}{x-\frac{1}{4}}} + \frac{1}{2} & \text{if } \frac{1}{4} < x \le \frac{1}{2}; \\
-\brac{\frac{1}{2} \sin 4 + 1}x + \frac{1}{2} \sin 4 + 1 & \text{if } \frac{1}{2} < x \le 1.
\end{cases}
\end{align*}

\begin{figure}[h!]
\centering
     \includegraphics[width=0.5\linewidth]{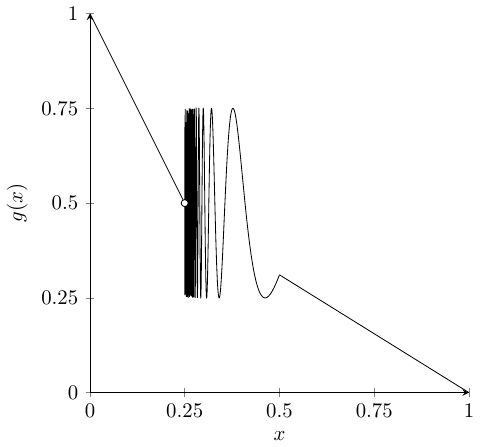}
     \caption{The graph of the function $g$.}
\end{figure}

Let $G$ be the graph of the function $g$, which additionally includes the line segment connecting the points $(1/4, 1/4)$ and $(1/4, 3/4)$ i.e.
\begin{align*}
    G = \left\{\brac{x, g(x)}\colon\; x \in I \setminus \left\{1/4\right\}\right\} \cup \left\{1/4\right\} \times \bracl{1/4, 3/4} \subset I^2.
\end{align*}

From the definition of $g$ and the fact that $G$ contains a part of rescaled topologist's sine curve we conclude that $G$ is closed and connected but not path-connected, and there is no path in $G$ connecting the points $(0, 1)$ and $(1, 0)$. 

For $i \in [n]$ let $\pi_i \colon I^n \to I$ be the projection onto the $i$th coordinate, and for $i, j \in [n]$ we define $\pi_{i, j} := (\pi_i, \pi_j) \colon I^n \to I^2$. For each $j \in [n-1]$, let $G_j = \pi_{j, j+1}^{-1}\bracl{G} \subset I^n$, and let $f_j \colon I^n \to \mathbb{R}$ be an~arbitrary smooth function satisfying $f_j^{-1}\bracl{\left\{0\right\}} = G_j$, which existence is guaranteed by Theorem \ref{extending}. Finally, let $f = (f_1, \ldots, f_{n-1}) \colon I^n \to \mathbb{R}^{n-1}$.

Let us suppose there exist $p \in \mathbb{R}^{n-1}$ and a path-connected subset $S \subset f^{-1}\bracl{\left\{p\right\}}$ that connects $i$th opposite faces of $I^n$ for some $i \in [n]$. Let $x^- \in S \cap I^n_{i, -}$ and $x^+ \in S \cap I^n_{i, +}$, which implies that $x^-_i = 0$ and $x^+_i = 1$. Let $S_i = \pi_{i, i+1}\bracl{S}$, which is a path-connected subset of $I^2$. We claim that $p_i = 0$. Indeed, let us assume, for the sake of contradiction, that $p_i \neq 0$. Then $S \subset f^{-1}\bracl{\left\{p\right\}} \subset f_i^{-1}\bracl{\left\{p_i\right\}}$, and consequently $S \cap G_i = \emptyset$. It follows that $S_i \cap G = \emptyset$ so $(x_i^\varepsilon, x_{i+1}^\varepsilon) = \pi_{i, i+1}(x^\varepsilon) \notin G$ for any $\varepsilon \in \left\{-, +\right\}$. Hence $x_{i+1}^- < 1$ and $x_{i+1}^+ > 0$. Let us define sets
\begin{align*}
&U_{<} = \left\{(x, y) \in \brac{I \setminus \left\{1/4\right\}} \times I\colon\; y < g(x)\right\} \cup \left\{1/4\right\} \times [0,\left.1/4\right), \\ &
U_{>} = \left\{(x, y) \in \brac{I \setminus \left\{1/4\right\}} \times I\colon\; y > g(x)\right\} \cup \left\{1/4\right\} \times \left(3/4,\right.1],
\end{align*}
i.e. the sets $U_{<}$ and $U_{>}$ are subsets of $I^2$ that lie, informally, strictly below and above the set $G$, respectively. Obviously these sets are open in $I^2$ and disjoint, and $I^2 = U_< \cup G \cup U_>$. Since $S_i \cap G = \emptyset$, then $S_i \subset I^2 \setminus G = U_< \cup U_>$ so $S_i = \brac{S_i \cap U_<} \cup \brac{S_i \cap U_>}$. But sets $S_i \cap U_<$ and $S_i \cap U_>$ are open in $S_i$, disjoint and nonempty since $\pi_{i, i+1}(x^-) \in S_i \cap U_<$ and $\pi_{i, i+1}(x^+) \in S_i \cap U_>$ so it contradicts the connectedness of $S_i$.

In that context, we obtain $p_i = 0$, which implies that $S \subset f_i^{-1}\bracl{\left\{0\right\}} = G_i$. Consequently, $S_i \subset G$, and it is clear that $\pi_{i, i+1}(x^-) = (0, 1)$ and $\pi_{i, i+1}(x^+) = (1, 0)$. However, since there is no path in $G$ connecting the points $(0, 1)$ and $(1, 0)$, this contradicts the fact that $S_i$ is path-connected.
\end{proof}

\section{Consequences of the continuous result}\label{sectionconseq}
We proved Theorem \ref{contthm} using Theorem \ref{maincombi}, which in turn relies on the version of the Steinhaus Chessboard Theorem (Theorem \ref{chessboardthm}). Although a proof of Theorem \ref{chessboardthm} is independent of Theorem \ref{contthm}, it turns out that, once Theorem \ref{contthm} has been established, Theorem \ref{chessboardthm} becomes a relatively straightforward consequence of it (see Theorem \ref{chessbymainth}). Moreover, Theorem \ref{contthm} also serves as a tool for proving the Brouwer Fixed Point Theorem (see Theorem \ref{brouwer}).

%As we indicated in the Introduction, the Brouwer Fixed Point Theorem is equivalent to the Poincaré-Miranda Theorem (see \cite[Theorem $4$]{mawhin} and \cite{kulpa}) and the Poincaré-Miranda Theorem can be proven by the $n$-dimensional Steinhaus Chessboard Theorem (see \cite[Theorem $3$]{chessboard}). Nevertheless, our aim is to provide a proof of the Brouwer Fixed Point Theorem directly using Theorem \ref{contthm}.

\begin{tw}\label{chessbymainth}
    The version of the Steinhaus Chessboard Theorem is a consequence of Theorem \ref{contthm}.
\end{tw}

\begin{proof}
Let $n, k \in \mathbb{N}, n \ge 2$ and $F \colon \mathcal{K}_k^n \to [n]$ be an arbitrary function. For $i \in [n-1]$ let $f_i \colon I^n \to \mathbb{R}$ be an arbitrary continuous function such that $f_i^{-1}\bracl{\left\{0\right\}} = \bigcup F^{-1}\bracl{\left\{i\right\}}$. We can choose such functions $f_i$ since the sets $\bigcup F^{-1}\bracl{\left\{i\right\}}$ are closed (see Theorem \ref{extending}). Let us define a continuous function $f \colon I^n \to \mathbb{R}^{n-1}$ as $f = (f_1, f_2, \ldots, f_{n-1})$. By virtue of Theorem \ref{contthm} there exist a point $p \in \mathbb{R}^{n-1}$ and a compact subset $S \subset f^{-1}\bracl{\left\{p\right\}}$ which connects $i$th opposite faces of $I^n$ for some $i \in [n]$.

Let us suppose that $p_j = 0$ for some $j \in [n-1]$. Thus $S \subset f_j^{-1}\bracl{\left\{p_j\right\}} = f_j^{-1}\bracl{\left\{0\right\}} =\bigcup F^{-1}\bracl{\left\{j\right\}}$. Let us denote $\mathcal{S}' = \left\{K \in F^{-1}\bracl{\left\{j\right\}}\colon\; S \cap K \neq \emptyset\right\} \subset F^{-1}\bracl{\left\{j\right\}}$. We obtain $S \subset \bigcup \mathcal{S}'$ and then $\bigcup \mathcal{S}'$ is connected by virtue of Lemma \ref{conn} and connects $i$th opposite faces of $I^n$ since $S$ does.

Now, let us suppose that $p_j \neq 0$ for every $j \in [n-1]$. Then $S \cap \bigcup F^{-1}\bracl{\left\{j\right\}} \subset f_j^{-1}\bracl{\left\{p_j\right\}} \cap f_j^{-1}\bracl{\left\{0\right\}} = \emptyset$ for every $j \in [n-1]$ so $S \subset \bigcup F^{-1}\bracl{\left\{n\right\}}$ and we proceed in the same way as above.
\end{proof}

\begin{tw}\label{brouwer}
    The Brouwer Fixed Point Theorem is a consequence of Theorem \ref{contthm}.
\end{tw}

\begin{proof}
    It is commonly known that it is enough to prove that there is no retraction from $I^n$ into $\partial I^n$ since then the fixed point property follows easily. For that, let us suppose conversely that such a retraction $r \colon I^n \to \partial I^n$ exists i.e. $r$ is continuous and $\restr{r}{\partial I^n}$ is the identity function.

    Let $g \colon \partial I^n \to \mathbb{R}^{n-1}$ be an arbitrary continuous function\footnote{For instance, we can take $g = \pi \circ h$ where $h \colon \partial I^n \to S^{n-1}$ is a homeomorphism and $\pi \colon S^{n-1} \to \mathbb{R}^{n-1}$ is the projection defined as $\pi\brac{x_1, \ldots, x_n} = \brac{x_1, \ldots, x_{n-1}}$.} such that $\abs{g^{-1}\bracl{\left\{p\right\}}} \le 2$ for every \hbox{$p \in \mathbb{R}^{n-1}$} and we define $f = g \circ r \colon I^n \to \mathbb{R}^{n-1}$. Since $f$ is continuous, we can use Theorem \ref{contthm} so there exist a~point $p \in \mathbb{R}^{n-1}$ and a compact subset $S \subset f^{-1}\bracl{\left\{p\right\}}$ which connects some opposite faces of $I^n$. Thus, in particular $\abs{S \cap \partial I^n} \ge 2$. Since $r$ is a retraction we have $\restr{f}{\partial I^n} = g$ so $g\bracl{S \cap \partial I^n} = f\bracl{S \cap \partial I^n} \subset f\bracl{S} = \left\{p\right\}$ and then $S \cap \partial I^n \subset g^{-1}\bracl{\left\{p\right\}}$. But $2 \le \abs{S \cap \partial I^n}$ and $\abs{g^{-1}\bracl{\left\{p\right\}}} \le 2$ so $S \cap \partial I^n = g^{-1}\bracl{\left\{p\right\}} = \left\{x_0, x_1\right\}$ for some points $x_0 \neq x_1$. We have $g \circ r \bracl{S} = f\bracl{S} = \left\{p\right\}$ so $r\bracl{S} \subset g^{-1}\bracl{\left\{p\right\}} = \left\{x_0, x_1\right\} = S \cap \partial I^n = r\bracl{S \cap \partial I^n} \subset r\bracl{S}$. Thus $r\bracl{S} = \left\{x_0, x_1\right\}$ and then $S = \brac{r^{-1}\bracl{\left\{x_0\right\}} \cap S} \cup \brac{r^{-1}\bracl{\left\{x_1\right\}} \cap S}$. On the other hand, $x_0 \in r^{-1}\bracl{\left\{x_0\right\}} \cap S$ and $x_1 \in r^{-1}\bracl{\left\{x_1\right\}} \cap S$ so both sets $r^{-1}\bracl{\left\{x_0\right\}} \cap S$ and $r^{-1}\bracl{\left\{x_1\right\}} \cap S$ are nonempty, disjoint and closed in $S$. It contradicts the connectedness of $S$.
\end{proof}
\medskip

\textbf{Acknowledgements:} We would like to cordially thank Jarosław Grytczuk for pointing out the notion of clustered chromatic numbers and the community of Mathematics Stack Exchange for valuable suggestions regarding Proposition \ref{optimalconst1}. For the first author, this research was funded in whole or in part by {\it Narodowe Centrum Nauki}, grant 2021/42/E/ST1/00162.

\bibliography{chess_level_sets}{}

@article{chessboard,
  title={An $n$-dimensional version of {S}teinhaus' chessboard theorem},
  author={Tkacz, P. and Turza{\'n}ski, M.},
  journal={Topol. Its Appl.},
  volume={155},
  number={4},
  pages={354--361},
  year={2008},
  publisher={Elsevier}
}

@article{clusterednumber,
  title={Defective and {C}lustered {G}raph {C}olouring},
  author={Wood, D. R.},
  journal={Electron. J. Comb.},
  year={2012}
}

@book{kuratowski,
  title={Topology: Volume II},
  author={Kuratowski, K.},
  year={1968},
  publisher={Academic Press and PWN}
}

@article{Czarnecki,
author = {Czarnecki, A. and Kulczycki, M. and Lubawski, W.},
journal = {Ann. Polon. Math.},
number = {2},
pages = {189-191},
title = {On the connectedness of boundary and complement for domains},
volume = {103},
year = {2011}
}

@article{kulpa,
  title={The {P}oincar{\'e}-{M}iranda {T}heorem},
  author={Kulpa, W.},
  journal={Am. Math. Mon.},
  volume={104},
  number={6},
  pages={545--550},
  year={1997},
  publisher={Taylor \& Francis}
}

@book{lee,
  title={Introduction to {S}mooth {M}anifolds},
  author={Lee, J. M.},
  year={2003},
  publisher={Springer}
}

@book{rotman,
  title={An {I}ntroduction to {A}lgebraic {T}opology},
  author={Rotman, J. J.},
  volume={119},
  year={2013},
  publisher={Springer}
}

@book{burago,
  title={A {C}ourse in {M}etric {G}eometry},
  author={Burago, D. and Burago, Y. and Ivanov, S.},
  volume={33},
  year={2001},
  publisher={AMS}
}

@book{whyburn,
  title={Dynamic {T}opology},
  author={Whyburn, G. and Duda, E.},
  year={2012},
  publisher={Springer}
}

@article{kulpaparametric,
  title={Parametric extension of the {P}oincar{\'e} theorem},
  author={Kulpa, W. and Socha, L. and Turza{\'n}ski, M.},
  journal={Acta Univ. Carolin. Math. Phys.},
  volume={41},
  number={2},
  pages={39--46},
  year={2000},
  publisher={Charles University in Prague}
}

@article{miranda,
  title={Un osservazione su una teorema di {B}rouwer},
  author={Miranda, C.},
  journal={Boll. Unione Mat. Ital 3},
  pages={527},
  year={1940}
}

@book{kalejdoskop,
  title={Kalejdoskop matematyczny},
  author={Steinhaus, H.},
  year={1954},
  publisher={PWN}
}

@article{combi5,
title = {A combinatorial analog of a theorem of {F}. {J}. {D}yson},
author = {Jayawant, P. and Wong, P.},
journal = {Topol. Its Appl.},
volume = {157},
number = {10},
pages = {1833-1838},
year = {2010}
}

@article{combi1,
  title={Equilibrium theorem as the consequence of the {S}teinhaus chessboard theorem},
  author={Turzański, M.},
  journal={Topol. Proc.},
  volume={25},
  year={2000}
}

@article{combi3,
title = {Algorithms for finding connected separators between antipodal points},
journal = {Topol. Its Appl.},
volume = {154},
number = {18},
pages = {3156-3166},
year = {2007},
author = {Boroński, J. P. and Minc, P. and Turzański, M.}
}

@article{combi4,
  title={On approximation of asymmetric separators of the $n$-cube},
  author={Boro{\'n}ski, J. P. and Turza{\'n}ski, M.},
  journal={Fixed Point Theory Appl.},
  volume={2012},
  pages={1--7},
  year={2012},
  publisher={Springer}
}

@article{combi2,
author = {Kulpa, W. and Turzański, M.},
year = {2001},
month = {01},
pages = {},
title = {A combinatorial theorem for a symmetric triangulation of the sphere ${S}^2$},
volume = {42},
journal = {Acta Univ. Carolin. Math. Phys.}
}

@article{freund,
title = {A constructive proof of {T}ucker's combinatorial lemma},
journal = {J. Comb. Theory, Ser. A},
volume = {30},
number = {3},
pages = {321-325},
year = {1981},
issn = {0097-3165},
doi = {https://doi.org/10.1016/0097-3165(81)90027-3},
url = {https://www.sciencedirect.com/science/article/pii/0097316581900273},
author = {Freund, R. M. and Todd, M. J.}
}

@article{combi6,
  title={On a discrete version of the antipodal theorem},
  author={Oleszkiewicz, K.},
  journal={Fund. Math.},
  volume={151},
  number={2},
  pages={189--194},
  year={1996}
}

@article{kidawa,
title = {The cube-like complexes and the {P}oincaré–{M}iranda theorem},
journal = {Topol. Its Appl.},
volume = {196},
pages = {198-207},
year = {2015},
issn = {0166-8641},
author = {Kidawa, M. and Tkacz, P.}
}

@article{waksman,
  title={A theorem on level lines of continuous functions},
  author={Waksman, Z. and Wasilewsky, J.},
  journal={Isr. J. Math.},
  volume={27},
  pages={247--251},
  year={1977},
  publisher={Springer}
}

@book{matouvsek,
  title={Using the Borsuk-Ulam theorem: lectures on topological methods in combinatorics and geometry},
  author={Matou{\v{s}}ek, J. and Bj{\"o}rner, A. and Ziegler, G. M.},
  year={2003},
  publisher={Springer}
}
\bibliographystyle{unsrt}
\smallskip
	{\small Micha{\l} Dybowski}\\
	\small{Faculty of Mathematics and Information Science,}\\
	\small{Warsaw University of Technology,}\\
	\small{Pl. Politechniki 1, 00-661 Warsaw, Poland} \\
	{\tt michal.dybowski.dokt@pw.edu.pl}\\
	\\
	{\small Przemys{\l}aw  G\'orka}\\
	\small{Faculty of Mathematics and Information Science,}\\
	\small{Warsaw University of Technology,}\\
	\small{Pl. Politechniki 1, 00-661 Warsaw, Poland} \\
	{\tt przemyslaw.gorka@pw.edu.pl}\\
\end{document}